\documentclass[10pt]{article}
\usepackage[utf8]{inputenc}
\usepackage{amsmath}
\usepackage{amssymb}
\usepackage{amsthm}
\usepackage{comment}
\usepackage{graphicx}
\usepackage[margin=1.1 in]{geometry}
\usepackage{xcolor}

\usepackage{enumitem}

\newtheorem{thm}{Theorem}[section]
\newtheorem{lemma}[thm]{Lemma}
\newtheorem{prop}[thm]{Proposition}

\newtheorem{cor}[thm]{Corollary}
\theoremstyle{definition}
\newtheorem{defin}[thm]{Definition}
\newtheorem{ex}[thm]{Example}
\newtheorem{rem}[thm]{Remark}

\newtheorem{assu}[thm]{Assumption}
\newtheorem{que}[thm]{Question}
\newtheorem{claim}[thm]{Claim}

\def\calC{\mathcal{C}}

\def\calE{\mathcal{E}}

\def\calN{\mathcal{N}}

\def\calQ{\mathcal{Q}}

\def\calV{\mathcal{V}}

\def\calX{\mathcal{X}}

\newcommand{\N}{\mathbb{N}}
\newcommand{\R}{\mathbb{R}}
\newcommand{\C}{\mathbb{C}}

\newcommand{\Z}{\mathbb{Z}}

\newcommand{\E}{\mathbb{E}}
\newcommand{\T}{\mathbb{T}}

\def \eps {\varepsilon}

\def\br#1{\left(#1\right)}
\def\brb#1{\left[#1\right]}

\def\tr{\textup{tr}}

\title{On the limiting behaviour of arithmetic toral eigenfunctions}
\author{Riccardo W. Maffucci and Alejandro Rivera}
\date{}
\newcommand{\Addresses}{{

  \bigskip\footnotesize
  
  R.W.~Maffucci, \textsc{EPFL, MA SB Batiment 8, Lausanne, Switzerland.}\par\nopagebreak
  \texttt{riccardo.maffucci@epfl.ch}
  
  \bigskip
  
  A.~Rivera, \textsc{EPFL, MA SB Batiment 8, Lausanne, Switzerland.}\par\nopagebreak
  \texttt{alejandro.rivera@epfl.ch}
}}

\begin{document}

\maketitle

\begin{abstract}

We consider a wide class of families $(F_m)_{m\in\N}$ of Gaussian fields on $\T^d=\R^d/\Z^d$ defined by 
\[
F_m:x\mapsto \frac{1}{\sqrt{|\Lambda_m|}}\sum_{\lambda\in\Lambda_m}\zeta_\lambda e^{i2\pi\langle \lambda,x\rangle}
\]
where the $\zeta_\lambda$'s are independent standard normals and $\Lambda_m$ is the set of solutions $\lambda\in\Z^d$ to the equation $p(\lambda)=m$ for some fixed elliptic polynomial $p$ with integer coefficients. The case $p(x)=x_1^2+\dots+x_d^2$ amounts to considering a random Laplace eigenfunction whose law is sometimes called the \textit{arithmetic random wave} and has been studied in the past by many authors. In contrast, we consider three classes of polynomials $p$: a certain family of positive definite quadratic forms in two variables, all positive definite quadratic forms in three variables except the multiples of $x_1^2+x_2^2+x_3^2$, and a wide family of polynomials in many variables.\\

For these three classes of polynomials, we study the $(d-1)$-dimensional volume $\calV_m$ of the zero set of $F_m$. We compute the asymptotics, as $m\to+\infty$ along certain well chosen subsequences of integers, of the expectation and variance of $\calV_m$. Moreover, we prove that in the same limit, $\frac{\calV_m-\E[\calV_m]}{\sqrt{\textup{Var}(\calV_m)}}$ converges to a standard normal.\\

As in previous analogous works on this topic for the arithmetic random wave, a very general method reduces the problem of these asymptotics to the study of certain arithmetic properties of the sets of solutions to $p(\lambda)=m$. More precisely, we need to study the number of such solutions for a fixed $m$, as well as the number of quadruples of solutions $(\lambda,\mu,\nu,\iota)$ satisfying $\lambda+\mu+\nu+\iota=0$, a.k.a. $4$-correlations, and the rate of convergence of the (rescaled) counting measure of $\Lambda_m$ towards a certain limiting measure on the hypersurface $\{p(x)=1\}$. To this end, we use many previous results on this topic but also prove a new estimate on correlations which may be of independent interest.
\end{abstract}

\tableofcontents

\section{Introduction}
\subsection{Nodal sets of random eigenfunctions}

The Laplace-Beltrami operator on the flat torus $\T^d=\R^d/\Z^d$ has a discrete spectrum and the dimension of the eigenspaces $E_m$ increases with the eigenvalue $m>0$. Given $\psi_m\in E_m\setminus\{0\}$, the nodal set of $\psi_m$, i.e. $\psi_m^{-1}(0)$, is the union of smooth hypersurface of $\T^d$, and a subset of lower Hausdorff dimension \cite{cheng1}. The geometry and topology of this nodal set may vary with $\psi_m$ but is subject to some restrictions depending on $m$, both for small values of $m$ and in the asymptotic regime $m\to +\infty$, which is the focus of the present document (see for instance \cite{zelditch_survey} for a survey of results on this topic).\\

By the real analytic case of Yau's conjecture \cite{yau982,yau993}, we know that uniformly in the choice of non-zero $\psi_m$,
\begin{equation}
\label{yau}
m^{1/2}\ll \text{Vol}(\{x\in\T^d : \psi_m(x)=0\}) \ll m^{1/2}\, .
\end{equation}

In an effort to understand the variation of this quantity within these constraints, several authors have studied the behaviour of $F_m^{-1}(0)$ where $F_m$ is a random element of $E_m$ (see for instance \cite{berard_1985}, \cite{rudwi2}, \cite{krkuwi}, \cite{mprw00}, \cite{benmaf}, \cite{cammar}, \cite{chelaa}). The two most popular probability measures, chosen because they respect the symmetries of the model, are defined as follows. The $L^2$-scalar product endows $E_m$ with a Euclidean structure. The random function $F_m$ is usually chosen to be either uniform on the unit sphere of $E_m$, or a standard Gaussian vector on $E_m$, i.e., with density $\frac{1}{(2\pi)^{\textup{dim}(E_m)/2}}e^{-\frac{1}{2}|v|_{L^2}^2}$ with respect to the Lebesgue measure $dv$ on $E_m$ (here $|\cdot|_{L^2}$ denotes the $L^2$-norm). Since the Gaussian measure is rotation invariant and the nodal set of $F_m$ is invariant by scaling of $F_m$ by positive constants, the law of $F_m^{-1}(0)$ is the same, whichever choice one makes between uniform and Gaussian.\\

In the present paper, we will focus on the study of the volume of the nodal set $\calV_m=\textup{Vol}(F_m^{-1}(0))$. This volume may be defined as the $(d-1)$-dimensional Hausdorff measure of $F_m^{-1}(0)$ or, equivalently, as the Riemannian volume measure whenever $F_m^{-1}(0)$ is smooth (which, in the cases we are interested in, will occur with probability one).\\

\subsection{State of the art on nodal length of random Laplace eigenfunctions on the torus}\label{ss:state_of_the_art}
In \cite{rudwi2} Rudnick and Wigman adapted an argument by Bérard (see \cite{berard_1985}) to show that, s $m\to+\infty$ along a well chosen sequence of density one among the integers $m$ for which $E_m\neq\{0\}$,
\begin{equation}
\label{rw2008}
\mathbb{E}[\mathcal{V}_m]=\sqrt{\frac{4\pi}{d}}\frac{\Gamma\left(\frac{d+1}{2}\right)}{\Gamma\left(\frac{d}{2}\right)}m^{1/2}
\end{equation}
which is consistent with Yau's conjecture \eqref{yau}. In \cite{krkuwi}, Krishnapur, Kurlberg and Wigman showed the precise asymptotic behaviour of the variance in the case $d=2$. For any subsequence of energies $(m_k)_k$ such that the multiplicities $\mathcal{N}_{m_k}\to\infty$, one has \cite[Theorem 1.1]{krkuwi}
\begin{equation}
\label{length}
\text{Var}(\mathcal{V}_m)=c_{m_k}\frac{m_k}{\mathcal{N}_{m_k}^2}(1+o(1)),
\end{equation}
where $\calN_m=\textup{dim}(E_m)$, and the positive real numbers $c_{m_k}$ depend on the \textit{limiting angular distribution} of $\Lambda_{m_k}$ -- the asymptotics are \textit{non-universal} \cite[section 1.2]{krkuwi}. The order of magnitude of \eqref{length} is smaller than the previously conjectured $\frac{m}{\calN_m}$ \cite{rudwi2}: remarkably, terms of order $\frac{m}{\calN_m}$ in the nodal length variance cancel out perfectly. This effect was called \textit{arithmetic Berry cancellation} in \cite[section 1.6]{krkuwi}, after ``Berry's cancellation phenomenon'' \cite{berry2,wispha}.\\

Subsequently it was shown \cite{mprw00} that the limiting nodal length \textit{distribution} is \textit{non-Gaussian}, and non-universal, depending again on the angular distribution of the lattice points \cite[Theorem 1.1]{mprw00}.\\

The three-dimensional case $p=x_1^2+x_2^2+x_3^2$ was investigated by Benatar and the first author \cite{benmaf}. As $m\to\infty$, $m\not\equiv 0,4,7 \pmod 8$, we have \cite[Theorem 1.2]{benmaf}
\begin{equation}
\label{eqn:benmaf}
\text{Var}(\calV_m)=\frac{m}{\mathcal{N}_m^2}\cdot\left[\frac{32}{375}+O\left(\frac{1}{m^{1/28-o(1)}}\right)\right].
\end{equation}
The congruence assumption is natural, as detailed in section \ref{sec:countlp}. The order of magnitude matches that of the $2$-dimensional case: $\T^3$ exhibits arithmetic Berry cancelation as well. In contrast to $d=2$, the leading term does not fluctuate, as lattice points on spheres {\em equidistribute in the limit} (see section \ref{sec:el23}).\\

Subsequently, Cammarota \cite{cammar} found that the limiting distribution of the nodal volume is non-Gaussian \cite[Theorem 1]{cammar},
\begin{equation}
\label{eqncammar}
\lim\frac{\mathcal{V}_m-\mathbb{E}[\mathcal{V}_m]}{\sqrt{\text{Var}(\mathcal{V}_m)}}=^d \frac{5-\chi}{\sqrt{10}}
\end{equation}
wherer the limit is taken as $m\to\infty$, $m\not\equiv 0,4,7 \pmod 8$, where $\chi$ is a chi-square r.v. with $5$ degrees of freedom, and $=^d$ denotes convergence in distribution.\\

Recently Cherubini and Laaksonen \cite{chelaa} gave an upper bound for the variance when $d\geq 4$,
\[\text{Var}(\calV_m)=O(m\calN_m^{-1-\alpha(d)+\epsilon}),\]
where $\alpha(4)=2/3$ and $\alpha(d)=2/(d-2)$ for $d\geq 5$. In particular, Berry cancellation is observed also for $d\geq 4$.\\

In the present manuscript we consider a wide family of Gaussian fields on $\T^d$ defined as sums of random waves determined by a certain polynomial equation replacing $\lambda_1^2+\dots+\lambda_d^2=m$. We study the asymptotic behaviour of the expectation and variance of $\calV_m$ as $m\to+\infty$ along certain well chosen subsequences and determine the rescaled limiting law. We find that in most cases Berry cancellation does not occur.

\subsection{Setting and main results}\label{ss:main_results}

Throughout this article, we will consider the torus $\T^d:=\R^d/\Z^d$ with the metric induced by the Euclidean metric on $\R^d$. We will consider $p\in\R[X_1,\dots,X_d]$ a $d$-variate polynomial homogeneous of degree $2k$. We will assume that $p$ is \textit{elliptic}, i.e., such that for each $\xi\in\R^d$, $p(\xi)\geq 0$, with equality if and only if $\xi=0$. In particular, the set $\Sigma_p=\{\xi\in\R^d\, :\, p(\xi)=0\}$ is a smooth compact hypersurface of $\R^d$. Our results will depend on the value of $d\in\N$, and on the choice of polynomial $p$.\\

The operator $p(-i\nabla)$ acting on $C^\infty(\T^d)$ has a discrete spectrum containing only eigenvalues and its eigenspaces are finite-dimensional. More precisely, its eigenvalues $m\in\R$ are those for which the following set is non-empty:
\[
\Lambda_{m}:=\{\lambda\in(1/m^{2k})\Z^d\, :\, p(\lambda)=1\}\, .
\]
For each such eigenvalue $m$, the corresponding eigenspace is generated by the $x\mapsto e(m^{1/2k}\langle\lambda,x\rangle)$ where for each $t\in\R$, we write $e(t):=e^{2\pi it}$ and where $\lambda$ ranges over $\Lambda_{m}$. In particular, its dimension is exactly $\calN_{m}:=|\Lambda_{m}|$.\\

If we equip this eigenspace with the $L^2(\T^d)$ scalar product and consider a random vector in this space whose law is the standard Gaussian measure induced by this scalar product. We then divide this vector by $\sqrt{\calN_m}$ and obtain an a.s. smooth Gaussian field $(F_m(x))_{x\in\T^d}$ on $\T^d$, which can be described as follows:
\begin{equation}\label{eq:field_definition}
F_{m}(x)=\frac{1}{\sqrt{\calN_m}}\sum_{\lambda\in\Lambda_m}\zeta_\lambda e(m^{1/2k}\langle\lambda,x\rangle)
\end{equation}
where $(\zeta_\lambda)_{\lambda\in\Lambda_m}$ are complex standard normals defined on a common probability space $(\Xi,\mathcal{F},\mathbb{P})$, independent save for the relation $\zeta_{-\lambda}=\overline{\zeta_\lambda}$.\\

We are interested in the properties of the \textit{nodal set} of $F_m$, i.e., the set $\{F_m(x)=0\}$ as $m\to +\infty$ along some subsequence of eigenvalues $\mathfrak{S}$. The most basic assumption we will make on $\mathfrak{S}$ is the following.
\begin{assu}\label{as:ample}
We assume that $\mathfrak{S}$ is such that $\lim_{m\in\mathfrak{S},m\to+\infty}\calN_m=+\infty$.
\end{assu}
This implies that for $m$ large enough, $\Lambda_m$ must contain two non-colinear vectors. In particular, for each $x$, the eigenvector $(F_m(x),\nabla F_m(x))$ is non-degenerate, which implies that the nodal set of $F_m$ is a.s. a smooth hypersurface of $\T^d$ (see for instance Proposition 6.12 of \cite{azawsc}). In the present work, we will study the asymptotic properties as $m\to+\infty$ in $\mathfrak{S}$ of the volume of the nodal set of $F_m$, or \textit{nodal volume}:
\[
\calV_m:=\textup{Vol}\{x\in\T^d\, :\, F_m(x)\}
\]
Though the random vector $\nabla F_m(0)$ is non-degenerate, it could degenerate asymptotically, which would qualitatively amount to the field concentrating around one frequency direction\footnote{By this we mean that the $\lambda\in\Lambda_m$, which dictate the frequency and direction of oscillation of the waves $e(\langle \lambda,\cdot\rangle)$, would concentrate along a given direction in $\R^d$.}. To avoid such pathological cases, we add an assumption on the eigenvalues of the covariance matrix of $\nabla F_m(0)$. To do so we introduce the scaling parameter
\begin{equation}\label{eq:l_m_def}
L_m=4\pi^2m^{1/k}
\end{equation}
which will naturally appear in the proofs.
\begin{defin}\label{def:omega}
Let $\tilde{\Omega}_m$ (resp. $\Omega_m$) be the covariance matrix of $\nabla F_m(0)$ (resp. $\frac{1}{\sqrt{L_m}}\nabla F_m(0)$). In particular, $\Omega_m=(1/L_m)\tilde{\Omega}_m$.
\end{defin}
We note that by definition, the matrix $\Omega_m$ is bounded from above uniformly in $m$.
\begin{assu}\label{as:nd_eigvals}
The sequence $\mathfrak{S}$ is such that the eigenvalues of $\Omega_m$ are bounded from below uniformly for $m\in\mathfrak{S}$.
\end{assu}
We will always work under Assumptions \ref{as:ample} and \ref{as:nd_eigvals} either by requiring explicitely or by ensuring that they hold. Some of our results will hold for every choice of $d$ and $p$. Others will require some additional assumptions. To express these, we introduce the \textit{degeneracy} of $p$, defined as follows. For each $x\in\R^d$, let $\nabla^2p(x)$ be the Hessian of $p$, i.e., the matrix $(\partial_{x_i}\partial_{x_j} p(x))_{i,j}$. Then, the degeneracy of $p$ is:
\begin{equation}\label{eq:degeneracy}
\overline{d}=d-\inf_{x\in\R^d\setminus\{0\}}\textup{rank}(\nabla^2 p(x))\, .
\end{equation}
We use the terminology \textit{non-degenerate} for $p$ whenever
\begin{equation}\label{eq:non_degenerate_polynomial}
\min\{d,2(d-\overline{d})\}>5\cdot 2^{2k+1}(2k-1).
\end{equation}
\begin{ex}
Assume that $p(x)=x_1^{2k}+\dots+x_d^{2k}$ with $k>0$. Then, $\nabla^2 p(x)$ is diagonal and its $j$th diagonal coefficient is $2k(2k-1)x_j^{2k-2}$ so $\overline{d}=0$ and $p$ is non-degenerate as long as
\begin{equation}\label{eq:diagonal_forms_nd}
d>5\cdot 2^{2k+1}(2k-1)\, .
\end{equation}
\end{ex}
\begin{rem}
\label{rem:41.1}
For ellipsoids ($2k=2$), $\nabla^2 p=v$ for fixed $v\in\R^d$ is simply $Ax=v$ where $A$ is the invertible matrix of the quadratic form $p$. In particular, $\overline{d}=0$ and $p$ is non-degenerate for $d\geq 41$.
\end{rem}

Our first additional assumption is the following:
\begin{assu}\label{as:weak}
Either $2k=2$ (i.e. $p$ is quadratic) and $d\in\{2,3\}$ or $p$ is non-degenerate.
\end{assu}

In order to state our results, we introduce two functionals over the set of symmetric positive matrices, one of which is scalar valued while the other is matrix valued. More precisely, for each positive definite $d\times d$ matrix $\Omega$, let 
\begin{equation}\label{eq:upsilon_psi_definition}
\Upsilon(\Omega)=\int_0^\infty\left(1-\det(I_d+t\Omega)^{-1/2}\right)\frac{dt}{t^{3/2}}\, ;\ \Psi(\Omega)=\int_0^\infty \det(I_d+t\Omega)^{-1/2}(I_d+t\Omega)^{-1}\frac{dt}{t^{1/2}}\, .
\end{equation}
One can easily check that the maps $\Omega\mapsto\Upsilon(\Omega)$ and $\Omega\mapsto\Psi(\Omega)$ are continuous on the space of positive definite $d\times d$ matrices. We obtain the following information about the asymptotic distribution of $\calV_m$ as $m\to+\infty$ along the subsequence $\mathfrak{S}$.
\paragraph{Expected nodal length.}
\begin{prop}[Expected nodal length]
\label{prop:expectation}
Let $d\in\N$ and $p\in\R[X_1,\dots,X_d]$ be elliptic homogeneous of degree $2k$. Then, as $m\to+\infty$ along the spectrum of $p(-i\nabla)$,
\label{propexpect}
\begin{equation}
\label{expect}
\mathbb{E}[\mathcal{V}_m]=\frac{\sqrt{L_m}}{2\pi}\Upsilon(\Omega_m)\, .
\end{equation}
\end{prop}
The proof of Proposition \ref{prop:expectation} may be found in section \ref{ss:integral_expressions} below. It is a direct application of the Kac-Rice formula.
\begin{rem}
If $p(x)=x_1^2+\dots+x_d^2$ then for each $m$, $\Omega_m=I_d/d$ and $\Upsilon(\Omega_m)=\sqrt{\frac{4\pi}{d}}\frac{\Gamma\br{\frac{d+1}{2}}}{\Gamma\br{\frac{d}{2}}}$ (see \eqref{eq:expectation_formula}) which is just $\mathcal{I}_d$ from Proposition 4.1 of \cite{rudwi2} (cf. \eqref{rw2008}). In general $\Upsilon(\Omega_m)$ may depend on $m$ but by Assumption \ref{as:nd_eigvals}, $\Upsilon(\Omega_m)$ is bounded from above and below by positive constants, independent on $m$.
\end{rem}

\paragraph{Variance upper bound.}

\begin{prop}[Variance upper bound]\label{prop:variance_upper_bound}
Fix $d\in\N$ and $p\in\R[X_1,\dots,X_d]$ elliptic homogeneous of degree $2k$ as well as $\mathfrak{S}$ satisfying Assumptions \ref{as:ample} and \ref{as:nd_eigvals}. Then, uniformly for $m\in\mathfrak{S}$,
\begin{equation}
\label{genbd}
\frac{\text{Var}(\calV_m)}{\E[\calV_m]^2}\ll_p\frac{1}{\calN_m}\, .
\end{equation}
\end{prop}
Proposition \ref{prop:variance_upper_bound} will be shown in section \ref{s:main_proofs}. Combining \eqref{genbd} with \eqref{expect}, we deduce, via the Chebychev-Markov inequality, that under the assumptions of Proposition \ref{prop:variance_upper_bound},
we have for all $\eps>0$ and $m\in\mathfrak{S}$,
\begin{equation}\label{eq:consec}
\mathbb{P}\left(\left|\frac{\mathcal{V}_m}{m^{1/k}}-\Upsilon(\Omega_m)\right|>\eps\right)\leq\frac{K}{\eps^2\calN_m}
\end{equation}
for some positive $K$ depending on $p$, which goes to zero as $m\to+\infty$ by Assumption \ref{as:ample}.

\paragraph{Variance asymptotics.}
Under the additional Assumption \ref{as:weak}, we can compute the main term in the asymptotics of the variance, along certain subsequences of integers.\\

In order to express our results, we introduce certain measures on the hypersurface $\Sigma_p=\{\xi\in\R^d\, :\, p(\xi)=1\}$. First, we denote by $dS_p$ the hypersuraface area measure. Then, we define a measure $d\sigma_p$ as follows.

\begin{defin}[Equidistribution measure]\label{def:equidistribution_measure}
We define $d\sigma_p$ to be the measure $\frac{dS_p}{Z_p|\nabla p|}$ where $Z_p>0$ is a normalizing constant so that the total mass of $\sigma_p$ is one. Equivalently, for each $\phi\in C^0(\Sigma_p)$,
\[
\int_{\Sigma_p}\phi(x)d\sigma_p(x):=\frac{1}{|S^{d-1}|}\int_{S^{d-1}} \phi\left(\frac{\omega}{p(\omega)^{1/2k}}\right)d\omega\, .
\]
where $d\omega$ is the Euclidean surface area measure on $S^{d-1}$
\end{defin}
\begin{ex}
If $p(x)=a_1x_1^2+\dots+a_dx_d^2$ then $d\sigma_p(x)$ is proportional to $\frac{dS_p(x)}{\sqrt{a_1^2x_1^2+\dots+a_d^2x_d^2}}$.
\end{ex}
Using the measure $d\sigma_p$ we define the matrix:
\begin{equation}\label{eq:omega_def}
\Omega:=\int_{\Sigma_p} x\otimes x^* d\sigma_p(x)\, .
\end{equation}
In other words, for each $i,j\in\{1,\dots,d\}$, $\Omega_{ij}=\int_{\Sigma_p}x_ix_jd\sigma_p(x)$. Since $d\sigma_p$ and $dS_p$ are mutually absolutely continuous with respect to each other, it is easy to check that $\Omega$ is invertible.

\begin{thm}[Variance asymptotics]\label{thm:variance}
Fix $d\in\N$ and $p\in\R[X_1,\dots,X_d]$ elliptic homogeneous of degree $2k$ satisfying Assumption \ref{as:weak}. For each $m\in\N$, define $\mathcal{E}_m$ whenever $\E[\calV_m]>0$, by
\begin{equation}
\label{var}
\frac{\textup{Var}(\mathcal{V}_m)}{\E[\calV_m]^2}=\frac{\pi}{\calN_m}\int_{\Sigma_p}\left(\Xi(x)-1\right)^2d\sigma_p(x)+\mathcal{E}_m\, .
\end{equation}
where for each $\lambda\in\Sigma_p$,
\[
\Xi(\lambda)=\Upsilon(\Omega)^{-1}\langle\lambda,\Psi(\Omega)\lambda\rangle\, .
\]
Here, $\Omega$ is defined as in \eqref{eq:omega_def}. Then, the remainder term $\mathcal{E}_m$ follows
\begin{enumerate}[label=(\Alph*)]
\item Assume that $2k=2$, $d=2$ and the discriminant of $p$ (see Appendix \ref{s:2D_equidistribution} for its definition) belongs to the set $\mathfrak{D}$ defined in Appendix \ref{s:2D_equidistribution}. Then, there exists a positive density sequence $\mathfrak{S}\subset\N$ such that uniformly for $m\in\mathfrak{S}$,
\[
\mathcal{E}_m=O\left(\calN_m^{-2}\right)\, ;
\]
\item If $2k=2$ and $d=3$ then there exists a positive density sequence $\mathfrak{S}\subset\N$ such that for each $\eps>0$, uniformly for $m\in\mathfrak{S}$,
\[
\mathcal{E}_m=O\left(\calN_m^{-1-\frac{1}{111}+\eps}\right)\textup{ and }\mathcal{E}_m=O\left(m^{-\frac{1}{2}-\frac{1}{222}+\eps}\right)\, ;
\]
\item If $p$ is non-degenerate (see \eqref{eq:non_degenerate_polynomial}), then, there exists some $\delta>0$ and a positive density sequence $\mathfrak{S}\subset\N$ such that
\[
\mathcal{E}_m=O\left(\calN_m^{-1-\delta}\right)\textup{ and }\mathcal{E}_m=O\left(m^{-(\frac{d}{2k}-1)(1+\delta)}\right)\, .
\]
\end{enumerate}
\end{thm}
\begin{rem}
\label{rem:41.2}
In particular, $\calE_m=\calN^{-\frac{2}{d-2}}$ for $2k=2$, $d\geq 41$.
\end{rem}
\begin{rem}
\begin{itemize}
\item In each case, the sequence $\mathfrak{S}$ is either determined explicitely or characterised by certain sufficient conditions in the various results cited or proved in section \ref{s:arithmetic} below. As explained in section \ref{ss:historical_context}, the exact expression of $\mathfrak{S}$ is related to certain congruence obstructions to finding solutions to $p(x)=m$, which depend on $p$ in various ways.
\item The leading term is non-negative and vanishes exactly when $\Xi(\lambda)=1$ for each $\lambda\in\Sigma_p$. This is equivalent to the fact that $\Psi(\Omega)^{-1/2}\Sigma_p$ be the unit sphere\footnote{If it is included in the unit sphere, since it is a closed compact hypersurface, it is equal to the unit sphere.}. This happens in particular when $p(x)=|x|^2$ as remarked in \cite{krkuwi} and \cite{benmaf} in dimensions $2$ and $3$ respectively.
\end{itemize}
\end{rem}

\paragraph{Distribution.}
Another main result of this work is the limiting \textit{distribution} of the nodal volume in this case, save when $\Psi(\Omega)^{-1/2}\Sigma_p$ is a sphere.

\begin{thm}
\label{thm:clt}
Fix $d\in\N$ and $p\in\R[X_1,\dots,X_d]$ elliptic homogeneous of degree $2k$ satisfying Assumption \ref{as:weak}. Assume also that $\Psi(\Omega)^{-1/2}\Sigma_p$ is not a sphere. Let $\mathfrak{S}\subset\N$ be the subset introduced in Theorem \ref{thm:variance}. Then, as $m\to+\infty$ along $\mathfrak{S}$, the random variable
\[
\frac{\mathcal{V}_m-\mathbb{E}[\mathcal{V}_m]}{\sqrt{\text{Var}(\mathcal{V}_m)}}
\]
converges in distribution to a standard normal $\calN(0,1)$.
\end{thm}

Regarding the condition that $\Psi(\Omega)^{-1/2}\Sigma_p$ is not a sphere, let us note first that whenever the degree of $p$ is $2k>2$, it is automatically satisfied since $\Sigma_p$ is an algebraic hypersurface of degree greater than two and $\Psi(\Omega)^{-1/2}$ is a linear map. In the case where $\Sigma_p$ is an ellipsoid, however, we have not found a general result to rule out cancellation. However, in section \ref{s:anticancellation} below, we study the family of ellipsoids $ax_1^2+x_2^2+x_3^2$ and prove that for $a\in\N$ large enough, $\Psi(\Omega)^{-1/2}\Sigma_p$ is not a sphere (see Proposition \ref{prop:anticancellation}).

\paragraph{Further remarks.}

\begin{cor}\label{cor:symmetric}
In addition to the assumptions as Theorem \ref{thm:variance}, assume that $p$ be symmetric polynomial, such that only even powers of the indeterminates appear among its monomials. Then
\begin{enumerate}[label=(\roman*)]
\item
The expected nodal volume is
\begin{equation}
\label{expectd4}
\E[\calV_m]=\sqrt{\frac{4\pi A_m}{d}}\frac{\Gamma\left(\frac{d+1}{2}\right)}{\Gamma\left(\frac{d}{2}\right)}\times m^{1/2k}
\end{equation}
where $A_m=\frac{1}{\calN_m}\sum_{\lambda\in\Lambda_m}|\lambda|^2$ satisfies, as $m\to+\infty$ along some positive density sequence $\mathfrak{S}\subset\N$,
\[
A_m=(1+o(1))\frac{1}{|S^{d-1}|}\int_{S^{d-1}}\frac{1}{p(\omega)^{1/2k}}d\omega\, .
\]
\item
For each $m\in\N$ such that $\E[\calV_m]>0$, let
\begin{equation}
\label{vard4}
\calE_m:=\frac{\text{Var}(\mathcal{V}_m)}{\E[\calV_m]^2}-\frac{\pi}{\calN_m}\times\frac{1}{|S^{d-1}|}\int_{S^{d-1}}\br{(1/d)p(\omega)^{-\frac{1}{k}}-1}^2d\omega\, .
\end{equation}
Then, in each case $(A)$, $(B)$ and $(C)$ of Theorem \ref{thm:variance}, $\calE_m$ satisfies the estimates given in Theorem \ref{thm:variance}.
\end{enumerate}
Moreover, as $m\to+\infty$ along the same subsequence, $\frac{\calV_m-\E[\calV_m]}{\sqrt{\text{Var}(\mathcal{V}_m)}}$ converges to a standard normal.
\end{cor}
\begin{proof}
 If $p$ satisfies the assumptions of the corollary, then $\Omega_m$ is diagonal and all of its diagonal coefficients are equal to $(A_m/d)$. The first point follows from Proposition \ref{prop:expectation}. The asymptotic behaviour of $A_m$ is due to the fact that the counting measure on $\Lambda_m$ converges to $\sigma_p$ under the assumptions of Theorem \ref{thm:variance} as explained in its proof. The second point follows from Theorems \ref{thm:variance} and \ref{thm:clt} by computing $\Upsilon((A_m/d)I_d)$ and $\Psi((A_m/d)I_d)$. By integrating by parts, we deduce that $\Psi((A_m/d)I_d)=\frac{1}{d}\Upsilon((A_m/d)I_d)$ and we express $\Upsilon((A_m/d)I_d)$ using \eqref{eq:upsilon_formula} to conclude.
\end{proof}
\begin{ex}
Assume that $p(x)=x_1^{2k}+\dots+x_d^{2k}$ for some $k>0$ such that $d>5\times 2^{2k+1}(2k-1)$. By \eqref{eq:diagonal_forms_nd}, $p$ satisfies Assumption \ref{as:weak}. Moreover, $p$ is symmetric and $\Omega_m=\frac{1}{d}I_d$. By Corollary \ref{cor:symmetric}, we deduce that with this choice of $p$, as $m\to+\infty$ along a positive density sequence of integers,
\begin{align*}
\E[\calV_m]&=\left(\frac{4\pi}{d}\frac{1}{|S^{d-1}|}\int_{S^{d-1}}\frac{1}{p(\omega)^{1/2k}}d\omega\right)^{1/2}\frac{\Gamma\left(\frac{d+1}{2}\right)}{\Gamma\left(\frac{d}{2}\right)}\times m^{1/2k}(1+o(1))\\
\frac{\text{Var}(\mathcal{V}_m)}{\E[\calV_m]^2}&=\frac{\pi}{\calN_m}\times\frac{1}{|S^{d-1}|}\int_{S^{d-1}}\br{(1/d)p(\omega)^{-\frac{1}{k}}-1}^2d\omega(1+o(1))\, .
\end{align*}
Moreover, as $m\to+\infty$ along the same subsequence, $\frac{\calV_m-\E[\calV_m]}{\sqrt{\text{Var}(\mathcal{V}_m)}}$ converges to a standard normal.
\end{ex}

For bivariate polynomials of any positive even degree we have the following.

\begin{prop}
	\label{prop:varcur}
	Let $p\in\R[X_1,X_2]$ be elliptic homogeneous of degree $2k$. Then
	\begin{equation}
	\label{eqn:varcur}
	\frac{\textup{Var}(\mathcal{V}_m)}{\E[\calV_m]^2}=\frac{\pi}{\calN_m^2}\sum_{\lambda\in\Lambda_m}(1-\Xi_m(\lambda))^2+O\left(\frac{1}{\calN_m^2}\right).
	\end{equation}
\end{prop}
The behaviour of the first summand on the RHS of \eqref{eqn:varcur} depends on 
the properties of $|\Lambda_m|$ as $m$ gets large. For degree $2k\geq 4$ these are not in general well understood -- see section \ref{sec:countlp}.
\begin{rem}
	Whenever the summation on the right-hand side dominates the error term, using Proposition \ref{prop:second_chaos} and reasoning as in the proof of Theorem \ref{thm:clt}, one can deduce that $\frac{\calV_m-\E[\calV_m]}{\sqrt{\text{Var}(\mathcal{V}_m)}}$ converges to a standard Gaussian.
\end{rem}

\paragraph{Future directions.}
There is a rich literature (e.g., \cite{becawi,cmwsh4,canhan,cantot,daesle,elhtot,totzel}) on geometric functionals of random fields, and on the asymptotic behaviour of high energy eigenfunctions. It would be interesting to study other geometric functionals for the fields \eqref{eq:field_definition}, in two or more dimensions: nodal components, critical points, excursion sets...\\

Moreover, as explained below Theorem \ref{thm:clt}, starting from the Berry cancellation condition we have established, we are unable to provide a tractable criterion to decide whether or not it occurs for general ellipsoids. This matter needs further clarification.\\

Finally, as the present paper shows and others suggest, any progress on the understanding of the arithmetic of $\Lambda_m$ for a given class of polynomials would yield results on $\calV_m$ and other geometric functionals of the field $F_m$.

\subsection{Notation and terminology}

For two positive functions $f(x),g(x)$ of the real variable $x$, the expression
\begin{equation*}
f\sim g \qquad\text{ means that }\qquad \lim_{x\to\infty}\frac{f(x)}{g(x)}=1.
\end{equation*}
The interchangeable notations (respectively Landau's and Vinogradov's)
\begin{equation*}
f=O(g), \quad\quad f\ll g
\end{equation*}
mean that $\exists c,x_0 : \forall x>x_0$ one has $f(x)\leq cg(x)$. In case $f\ll g\ll f$, we write $f\asymp g$. The expression
\begin{equation*}
f=o(g) \qquad\text{ means that }\qquad \lim_{x\to\infty}\frac{f(x)}{g(x)}=0.
\end{equation*}
We may also write an index, e.g. $f=O_a(g)$, to stress dependence on the quantity $a$.\\

Let $A'\subseteq A\subseteq\mathbb{Z}$. We say $A'$ has \textit{asymptotic density} $l$, $0\leq l\leq 1$ in $A$ if
	\begin{equation}
	\label{density}
	l=\lim_{X\to\infty}\frac{|\{n\in A' : n\leq X\}|}{|\{n\in A : n\leq X\}|}.
	\end{equation}

\subsection{Proof outline and plan of the paper}
Let us outline the proofs here. The method is of similar flavour to the case where $p$ is the equation of a sphere \cite{krkuwi,benmaf}. A mean centred stationary Gaussian field such as $F_m$ on the torus $\T^d$ may be completely described by its \textit{covariance function}
\begin{equation}
\label{rintro}
r_m(x-y):=\mathbb{E}[F_m(x)F_m(y)]\, .
\end{equation}
From \eqref{eq:field_definition} we get, for each $x\in\T^d$,
\begin{equation}
\label{r}
r_m(x)=\frac{1}{\calN_m}\sum_{\lambda\in\Lambda_m}e(m^{1/2k}\lambda\cdot x)\, .
\end{equation}

To understand the functional $\calV_m$, we will use the \textit{Kac-Rice formulas} (see Theorem 6.9 of \cite{azawsc}). Given a random field satisfying certain conditions, these formulas constitute a standard tool to compute moments of the measure of the zero set (see for instance Chapter 6 of \cite{azawsc}).\\

Their application requires understanding the (scaled) \textit{second intensity} $K_{2;m}$ of $F_m$ (to be defined in \eqref{K2tdef} and \eqref{eq:K2_def}). Section \ref{s:covariance_proofs} is thus dedicated to the following arguments and computations. We will express $K_{2;m}$ in terms of the conditional Gaussian expectation of the $2d$-dimensional vector
\[
(\nabla F_m(0), \nabla F_m(x))
\]
conditioned on $F_m(0)=F_m(x)=0$. Following a method introduced by Berry in \cite{berry2}, we will then rewrite this conditional expectation explicitly in terms of the covariance function \eqref{r} and its various first and second order derivatives at $x$ (see \eqref{eq:K2_def}). Next, we approximate these functions by their Taylor expansion around zero (see Lemma \ref{lemma:berry_s_method} for the general expansion applied here). This expansion is not valid as the function $r_m$ and its derivatives do not decay uniformly away from the diagonal. To deal with this issue, we will define a \textit{singular} set $S_m\subset\mathbb{T}^d$ (see Definition \ref{def:singular_set}), outside of which the expansion is valid and show that its volume is small enough that the integral of $K_{2;m}$ over $S_m$ is negligible in the overall computation.\\

Having integrated the approximation formula, the error terms are expressed in terms of the two following arithmetic quantities (see the \textit{arithmetic formula} Proposition \ref{prop:variance}). First, $\calN_m$ the cardinality of $\Lambda_m$ and then, $|\calC_{4;m}|$ the number of quadruples $(\lambda,\mu,\nu,\iota)\in\Lambda_m^4$ such that $\lambda+\mu+\nu+\iota=0$, which are called $4$-correlations. On the other hand, the main term is expressed as the sum for $\lambda\in\Lambda_m$, of the values of an explicit function of $\lambda$. Assuming that the counting measure on $\Lambda_m$ converges to a limiting measure $d\sigma_p$, the expression simplifies considerably and ensures that the leading term is indeed greater than the error terms.\\

Hence, all that remains is to estimate $\calN_m$, $|\calC_{4;m}|$ and the rate of equidistribution. This is very difficult to do in full generality. However, we are able to do so in three cases. The arguments are collected in section \ref{s:arithmetic}. First, based on an equidistribution result by Dias in \cite{dias00} as well as some estimates by Cilleruelo and Córdoba from \cite{cilcor}, we are able to cover the case where $\Sigma_p$ belongs to a certain class of ellipses. Next, applying various results collected in section 11.6 of \cite{iwanbk} we cover the case where $\Sigma_p$ is an ellipsoid in dimension three. Finally, by applying Birch's results from \cite{birc62}, we deal with a wide family of polynomials $p$ with a high number of variables with respect to its degree which we call non-degenerate (as defined in \eqref{eq:non_degenerate_polynomial}).\\

In comparison to the cases covered in the present text, because of Berry cancellation, the case where $\Sigma_p$ is either a circle or a sphere in $\R^3$ involves the analysis of $6$-correlations (see \cite{krkuwi} and \cite{benmaf} for further details).\\

In section \ref{s:chaos}, the nodal volume distribution is derived by considering the \textit{Wiener chaos expansion} of $\calV_m$. More precisely, we consider random variables formed by taking polynomials in the random variables $(F_m(x))_{x\in\T^d}$. We then define, for each $q\in\N$, $C_q$ to be the space of such polynomials of degree $q$ which are $L^2$-orthogonal to the polynomials of lower degree. The sum of the spaces $C_q$ is dense in $L^2$ so we can decompose $\calV_m$ as
\begin{equation}
\label{wiener}
\calV_m=\sum_{q=0}^{\infty}\calV_m[q]\, .
\end{equation}
The series converges in $L^2(\mathbb{P})$, where $\calV_m[q]$ is the orthogonal projection of $\calV_m$ onto $C_q$. In section \ref{s:chaos}, we will obtain an upper bound on $\textup{Var}(\calV_m)-\textup{Var}(\calV_m[2])$ in terms of the same arithmetic quantities presented above (see Proposition \ref{prop:second_chaos}). It then follows from the same arithmetic results as before that $\calV_m$ is close to $\calV_m[2]$ in $L^2$. It is then straightforward to deduce a central limit theorem for $\calV_m$. This constitutes a marked difference with the case of spheres \cite{mprw00,cammar}, where the limiting distribution is non-Gaussian. In \cite{mprw00,cammar}, Berry cancellation is tantamount to the second order projection vanishing: in these works, the fourth chaotic component dominates in the expansion \eqref{wiener}.

\subsection{Acknowledgements}
R.M. was supported by Swiss National Science Foundation project 200021\_184927.

\section{Establishing the variance asymptotic: arithmetic results}
\label{s:arithmetic}
In order to study the nodal volume variance asymptotic, we need to understand subtle arithmetic properties of the (projected) lattice point set $\Lambda_m$. The goal of section \ref{s:arithmetic} is to collect prior results and prove new ones, in order to obtain the information needed for the proof of our main results. Our method requires lower bounds on the lattice point number $\calN_m$, upper bounds on the number of four correlations $|\calC_{4;m}|$, i.e., quadruples of eigenvalues whose sum vanishes, and upper bounds on the rate of convergence of the counting measure on $\Lambda_m$ towards the equidistribution measure $d\sigma_p$ (see Definition \ref{def:equidistribution_measure}).\\

This information needs to be obtained in the three following cases. Either $\Sigma_p$ belongs to a certain family of ellipses, or $\Sigma_p$ is an ellipsoid in dimension three, or $p$ is non-degenerate (see \eqref{eq:non_degenerate_polynomial}).\\

In section \ref{ss:historical_context}, we give an overview of previous arithmetic results in the same flavour as those used here. The results used in the rest of the paper are collected in sections \ref{ss:arithmetic_2D_3D} and \ref{ss:arithmetic_hd}.

\subsection{Historical context}\label{ss:historical_context}

In this section, we present some previous results estimating the asymptotic behaviour of $\calN_m$ and $|\calC_{4;m}|$. Let us emphasise that the purpose of the present section is merely to contextualise the following ones. The rest of the manuscript does not refer to results presented in this section.

\subsubsection{Counting lattice points}
\label{sec:countlp}

In general, for simple reasons of congruence obstructions, there are arithmetic sequences of $m$ along which $\calN_m=0$. For instance, if $p(x)=x_1^2+x_2^2$, the values of $p$ are exactly integers with no prime factors equal to $3$ modulo $4$ appearing with an odd exponent. It is natural to restrict to $\mathfrak{S}_p'$ the sequence of $m$ for which $\calN_m>0$.
In the case where $\Sigma_p$ is an ellipse, for each $\eps>0$, on a sequence of $m$ of density one in $\mathfrak{S}_p'$, the number of lattice points satisfies (see \cite{cicoel})
\begin{equation}
\label{lpelli2}
\calN_m\ll_\eps m^\eps\, .
\end{equation}
Considering again $p(x)=x_1^2+x_2^2$, there remain pathological subsequences of $\mathfrak{S}'_p$, e.g. $\calN_m=8$ for each $m$ prime congruent to $1 \pmod 4$. However, for a subsequence $\mathfrak{S}_p$ of density $1$ \textit{within} $\mathfrak{S}_p'$, the number $\calN_m$ is not bounded above by any power of $\log(m)$, and in particular $\calN_m\to\infty$ along $\mathfrak{S}_p$. This is the \textbf{sequence of regular values} for the corresponding polynomial $p(x)=x_1^2+x_2^2$.\\

Similarly \cite{ruwiye,maff3d}, there are no integer solutions to $x_1^2+x_2^2+x_3^2=m=4^n(8l+7)$, where $n,l\in\N$ are given. Excluding these values of $m$, one is left with a subsequence of density $5/6$ in the naturals. To get a lower bound on $\calN_m$, we take the sequence $\mathfrak{S}_p$ satisfying $m\not\equiv 0,4,7 \pmod 8$, of density $5/8$, for which $m^{1/2-\eps}\ll\calN_m\ll m^{1/2+\eps}$. This is the sequence of regular values for $x_1^2+x_2^2+x_3^2$. It is well-known that every positive integer is the sum of $4$ squares. One needs only to insist that $m$ is not a power of $2$ to obtain the lower bound in $m^{1-\eps}\ll\calN_m\ll m^{1+\eps}$ \cite{iwniec,sarn90}. It is natural to take into account these congruence obstructions, and work along sequences of regular values.\\

For any positive $\eps$ in the case of ellipsoids in dimension $d\geq 3$, along regular values in the cases $d=3,4$, one has for all $\eps>0$ \cite{iwniec,fome86}
\begin{equation}
\label{lpelli}
m^{\frac{d-2}{2}-\eps}\ll \calN_m\ll m^{\frac{d-2}{2}+\eps}
\end{equation}
where we may take $\eps=0$ for $d\geq 5$. We remark in particular that there are no congruence obstructions for ellipsoids in $d\geq 5$: here $\mathfrak{S}_p=\N$.\\

Much less is known in general for $p$ of degree $K>2$ (in the present section we do not necessarily assume $K$ to be even). Even in the special case
\begin{equation*}
x^{K}+y^{K}+z^{K}=m, \qquad K\geq 3,
\end{equation*}
for a very long time there existed only upper bounds of the form \cite{hbanna,hbsurv}
\begin{equation}
\label{ternary}
\calN_m\ll_{K,\eps} m^{1/K+\eps}
\end{equation}
until Heath-Brown found \cite[Theorem 13]{hbanna}
\begin{equation*}
\calN_m\ll_\eps m^{\theta+\eps}, \qquad \theta=\frac{2}{\sqrt{K}}+\frac{2}{K-1},
\end{equation*}
improving \eqref{ternary} for $K\geq 8$. In the other direction there is only a $\Omega$-result due to Mahler \cite{mahler}: for  $K=3$, $\calN_m=\Omega(m^{1/12})$. Mahler's method cannot work for any other degree except possibly $K=5$ \cite{hbsurv}. In fact, it is expected that $\calN_m\ll_{K,\eps} m^{\eps}$ as soon as $K\geq 4$ \cite{hbanna,hbsurv}.

There is a general heuristic argument (cfr. \cite[\S 1.3]{hbsurv}, \cite[\S 1]{durusa}) that the lattice points should be approximately
\begin{equation}
\label{heu}
m^{(d-K)/K}
\end{equation}
as soon as $d>K$. The argument is based on the fact that $p$ takes values in $[-Hm,Hm]$ for some large positive $H$, and that there are approximately $m^{d/K}$ many vectors $x$ such that for every coordinate $x_j$ one has the bound $|x_j|\ll m^{1/K}$.\\

When $d=2$ and $p(x)-m$ is absolutely irreducible, Bombieri and Pila \cite{bopi89} showed that
\begin{equation*}
\calN_m\ll_\eps m^{1/K^2+\eps}
\end{equation*}
remarkably independent of $p$. More generally, in the irreducible case Pila \cite{pila95} found the bound
\begin{equation*}
\calN_m\ll_\eps m^{(d-2+1/K)/K+\eps}
\end{equation*}
also independent of $p$. There has been the recent improvement \cite[Theorem 4]{ccdn19}
\begin{equation*}
\calN_m\ll_\eps m^{(d-2)/K},
\qquad
d\geq 3, \ K\geq 5
\end{equation*}
assuming that the homogeneous part of highest degree of $p$ is irreducible.\\

These bounds are quite far from the heuristics \eqref{heu}. To our best knowledge, the only general method to obtain an asymptotic for this problem (or any lower bound, for that matter) is the Hardy-Littlewood circle method. It requires \textit{the number of variables to be much larger than the polynomial degree}. Birch established, via the circle method, an asymptotic formula for solutions to integer systems in many variables (see Theorem \ref{prop:birch_generic_systems} to follow). In particular, if $p$ is an elliptic polynomial of degree $2k$ in $d$ variables, satisfying
\[
d>2^{2k}(2k-1)\,
\]
then, uniformly for each $m\in\N$, either $\calN_m=0$ or
\begin{equation}\label{sharp}
\calN_m\asymp m^{\frac{d}{2k}-1}
\end{equation}
and \eqref{sharp} holds along an arithmetic progression $\mathfrak{S}_p$
(see Proposition \ref{prop:lattice_points_hd} below).
The case of ellipsoids \eqref{lpelli} is a special case of \eqref{sharp} when the dimension is large enough.\\

As a straightforward consequence, of independent interest, we have the following new sharp bounds for lattice points lying on hyperplanes.
\begin{cor}
\label{ori}
Let $p$ be an elliptic polynomial of degree $2k$ in $d$ variables, satisfying $d-1>2^{2k}(2k-1)$. Define $\Sigma=\Sigma_p=\{p(x)=1\}$. Consider the intersection $\Sigma'$ between $\Sigma$ and any hyperplane containing the origin. Then there is an arithmetic progression such that $\Sigma'$ contains an order of
\[
|\Lambda_m\cap\Sigma'|\asymp\calN_m^{1-\frac{1}{d-2k}}
\]
lattice points as $m\to\infty$ along this progression.
\end{cor}
\begin{proof}
The equation defining the hyperplane through the origin is $\langle v,x\rangle=0$ for some $v\in\R^d$. Solving for $x_d$ and substituting into $p$ yields
\begin{equation*}
q(x)=m
\end{equation*}
where $q$ is homogeneous of degree $2k$ in $d-1>2^{2k}(2k-1)$ variables. By \eqref{sharp}, along an arithmetic progression $\mathfrak{S}_q$,
\begin{equation*}
|\Lambda_m\cap\Sigma'|\asymp m^\frac{(d-1)-2k}{2k}.
\end{equation*}
Since $\mathfrak{S}_p$ and $\mathfrak{S}_q$ both contain $m$, their intersection is another arithmetic progression. Along this sequence,
\begin{equation*}
|\Lambda_m\cap\Sigma'|\asymp m^\frac{(d-1)-2k}{2k}\asymp\calN_m^\frac{d-1-2k}{d-2k}
\end{equation*}
via another application of \eqref{sharp}.
\end{proof}


\subsubsection{Correlations}
\label{sec:corr}
Similarly to \cite{krkuwi, benmaf}, the following problem arises naturally (from Proposition \ref{prop:variance}).
\begin{que}
Given $p$, how big is the set $\calC_{4;m}$ as $m\to\infty$?
\end{que}
For every $p,m$ one trivially has
\begin{equation*}
\calN_m^2\ll|\calC_{4;m}|\leq \calN_m^3
\end{equation*}
where the lower bound is due to the ``diagonal'' $4$-correlations $(\lambda,-\lambda,\mu,-\mu)$. We immediately derive the generic variance upper bound \eqref{genbd}. Once we have proven Proposition \ref{prop:variance}, if we wish to establish an asymptotic for the nodal volume variance, we need (at least) the improvement $|\calC_{4;m}|=o(\calN^3)$.

More generally, define the \textit{length $\ell$ correlations}, $\ell\geq 3$,
\begin{equation}
\label{lcorr}
\calC_{\ell;m}:=\left\{(\lambda_1,\dots,\lambda_{\ell})\in\Lambda^{\ell} :
\sum_{i=1}^{\ell}\lambda_i=0
\right\}.
\end{equation}
For $\ell$ even we easily find
\begin{equation*}
\calN_m^{\ell/2}\ll|\calC_{\ell;m}|\leq \calN_m^{\ell-1},
\end{equation*}
again by taking into account the correlations that cancel out in pairs.

Let us give a brief account on the existing literature. For lattice points on circles, the $4$-correlations are well understood due to an elementary observation of Zygmund \cite{zyg}: fixing $\lambda_1,\lambda_2$, the two circles centred at the origin and at $\lambda_1+\lambda_2$, both of radius $m^{1/2}$, intersect in at most two points, resulting in a finite number of choices for the remaining $\lambda_3,\lambda_4$. Therefore, for circles $|\calC_{4;m}|\asymp\calN_m^2$. The case of spheres, though, requires understanding the $6$-correlations as well, due to the higher order terms cancelling out (Berry cancellation phenomenon). The fact that $|\calC_{6;m}|=O(\calN_m^4)$ is straightforward again by Zygmund's observation. In \cite[Theorem 2.2]{krkuwi}, the non-trivial
\[ |\calC_{6;m}|=o(\calN_m^4) \]
was shown. Subsequently, Bombieri-Bourgain proved in \cite{bombou}, among other results of this flavour, that actually
\[ |\calC_{6;m}|=O(\calN_m^{7/2}). \]

Similarly, \cite[Theorems 1.6-1.7]{benmaf} shows that, for spheres $p=x_1^2+x_2^2+x_3^2$, the non-trivial bounds
\[ |\calX_{4;m}|=O(\calN_m^{7/4+\eps}), \qquad\qquad |\calC_{6;m}|=O(\calN_m^{11/3+\eps}) \]
hold, where $\calX_{4;m}\subseteq\calC_{4;m}$ are the correlations that do not cancel in pairs. These bounds are needed again because of Berry cancellation.\\

For spheres in dimension $d\geq 5$, the work of Bourgain and Demeter \cite{bode15} implies the upper bounds
\begin{equation}
\label{corrsphu}
|\calC_{\ell;m}|\ll\calN_m^{\ell-1-\frac{2}{d-2}}, \qquad d\geq 4, \  \ell\geq\frac{2(d-1)}{d-3}, \ \ell \text{ even}.
\end{equation}
For $4$-correlations, the bound holds as soon as $d\geq 5$. On the other hand, it is implicit in \cite[section 3.3]{benmaf} that for spheres in $d\geq 5$,
\begin{equation}
\label{corrsphl}
|\calC_{\ell;m}|\gg\calN_m^{\ell-1-\frac{2}{d-2}}, \qquad \ell \text{ even}.
\end{equation}

\subsection{The case of ellipsoids in dimensions two and three}\label{ss:arithmetic_2D_3D}
\label{sec:el23}

In the present subsection, we state the results we use in the case where $2k=2$ and $d\in\{2,3\}$.\\

\paragraph{Lattice point count and equidistribution for ellipses.}

Assume that $p$ is a bivariate quadratic form with integer coefficients. Then, we will use the fact that for each $m\in\N$, $m\geq 1$, the number $\calN_m$ of points $(x,y)\in\Z^2$ such that $p(x,y)=m$ satisfies the following estimates.\\

First, $\calN_m$ is bounded from above, uniformly in $p$. This lemma easily follows from a discussion in section II.A of \cite{cilcor} as explained in Appendix \ref{s:upper_bd_proof}.

\begin{lemma}\label{lemma:upper_bd_ellipses}
Assume that $p$ is a bivariate quadratic form with integer coefficients. For each $\eps>0$ there exists $C=C(\eps)<+\infty$ such that the following holds. Let $p$ be a bivariate quadratic form with integer coefficients. Then, for each $m\in\N$, $m\geq 1$,
\[
\calN_m\leq Cm^\eps\, .
\]
\end{lemma}

On the other hand, along a density one sequence, we use that $\calN_m\to+\infty$. More precisely, we use the following result from \cite{dias00}.
\begin{lemma}[\cite{dias00}, Lemmas 2 and 3]\label{lemma:lattice_points_2D}
Assume that $p$ is a bivariate quadratic form with integer coefficients. Then, for each $\eps>0$ there exists $\mathfrak{S}\subset\N$ a sequence of density one among the integers $m\in\N$ for which $\calN_m>0$ such that for each $m\in\mathfrak{S}$
\[
\calN_m\geq\left(\frac{1}{2}-\eps\right)\log\log(m)\, .
\]
\end{lemma}

We will also use the fact that the points in $\Lambda_m$ equidistribute in the following sense for a for a certain family of ellipses. We refer the reader to Appendix \ref{s:2D_equidistribution} for the definition of $\mathfrak{D}$ used in the following proposition.
\begin{prop}[Equidistribution for ellipses]\label{prop:equidistribution for ellipses}
Assume that $d=2$ and that the degree of $p$ is $2k=2$. Assume also that the positive definite quadratic form $p(x,y)=ax^2+bxy+cy^2$ is primitive and that the discriminant $-\delta=b^2-4ac$ is such that $\delta\in\mathfrak{D}$. Then, for each $\eps>0$, there exists a subset $\mathfrak{S}\subset\N$ of density one among the integers $m$ for which $\calN_m>0$ and a constant $C=(p,\eps)<+\infty$ such that for each $m\in\mathfrak{S}$ and each $\phi\in C^0(\Sigma_p)$,
\[
\left|\frac{1}{\mathcal{N}_m}\sum_{\lambda\in\Lambda_m}\phi(\lambda)-\int_{\Sigma_p}\phi(x)d\sigma_p(x)\right|\leq C\|\phi\|_{\infty}\left(\log(m)\right)^{-\frac{1}{2}\log\left(\frac{\pi}{2}\right)+\eps}\, .
\]
\end{prop}
\begin{proof}
Theorem \ref{thm:dias} applied with $m=T$ yields the result for indicator functions instead of continuous functions. But continuous functions on $\Sigma_p$ can be approximated uniformly by sums of indicator functions so the proposition follows immediately.
\end{proof}

In contrast to Proposition \ref{prop:equidistribution for ellipses}, Cilleruelo and Córdoba showed that there exist ellipses, and (density $0$) sequences of energies, such that the lattice points belong to arbitrarily short arcs.
\begin{prop}[{\cite[Theorem 2]{cicoel}}]
\label{cilmep}
For every $\eps>0$ and for every integer $n$, there exists an ellipse $x^2+y^2/a=m$ such that all its lattice points are on the arcs $|y/x|<\eps$, and $\mathcal{N}_m>n$.
\end{prop}

\paragraph{Lattice point count and equidistribution for ellipsoids in dimension three.}
The following theorem contains the estimates on $\calN_m$ and the equidistribution rate in the case where $p$ is an ellipsoid in dimension three.
\begin{thm}[\cite{iwanbk}, section 11.6]\label{thm:arithmetic_dimension_3}
Assume that $d=3$ and that the degree of $p$ is $2k=2$. Recall the definition of $d\sigma_p$ from Definition \ref{def:equidistribution_measure}. Let $\eps>0$. Then, there exists a positive density sequence of integers $\mathfrak{S}=\mathfrak{S}(\eps,p)$ such that, for each $m\in\mathfrak{S}$ the following holds.
\begin{enumerate}
\item There exists a constant $C=C(p,\eps)<+\infty$ such that
\begin{equation}\label{eq:N_m_ellipsoids}
(1/C)m^{\frac{1}{2}-\eps}\leq \calN_m\leq C m^{\frac{1}{2}+\eps}\, .
\end{equation}
\item There exists $C=C(p,\eps)$ such that for each $\phi\in C^0(\Sigma_p)$ 
\[
\left|\frac{1}{\calN_m}\sum_{\lambda\in\Lambda_m}\phi(\lambda)-\int_{\Sigma_p} \phi(x) d\sigma_p(x)\right|\leq C\|\phi\|_\infty m^{-1/222+\eps}\, .
\]
\end{enumerate}
\end{thm}
\begin{proof}
For the most part, these results are presented in \cite{iwanbk} section 11.6. Let us simply add that since polynomials are dense in $C^0(\Sigma_p)$ for the topology of uniform convergence in $\Sigma_p$, we may extend the equidistributions from \cite{iwanbk} to continuous functions immediately.
\end{proof}

\paragraph{Four correlations in dimensions two and three.}
The following proposition is an upper bound on the number of four correlations in dimension two and for ellipsoids in dimension three. Crucially, in dimension three, it relies on bounds for the number of lattice points on ellipses that are \textit{uniform on the ellipse} (see Lemma \ref{lemma:upper_bd_ellipses} above).
\begin{prop}\label{prop:four_correlations_dim_2_and_3}
Let $p$ be a homogeneous $d$-variate polynomial, of degree $2k$. For each $\eps>0$ there exists $C=C(\eps)<+\infty$ such that for each $m\in\mathfrak{S}$,
\begin{equation*}
|\calC_{4;m}|\leq \begin{cases}
C\calN_m^{3-1}, & d=2,\\
C\calN_m^{3-(1-\eps)}
& 2k=2, \ d=3,
\end{cases}
\end{equation*}
\end{prop}
\begin{rem}
As explained in section \ref{sec:corr}, we actually have $|\calC_{4;m}|\geq \calN_m^2$. Consequently, for $d=2$ the bound is optimal up to constants and for $2k=2$, $d=3$ it is tight up to arbitrarily small powers of $\calN_m$.
\end{rem}
\begin{proof}[Proof of Proposition \ref{prop:four_correlations_dim_2_and_3}]
In dimension $d=2$, two smooth distinct compact connected curves may intersect in at most $(2k)^2$ points, where $2k$ is the degree. Therefore, once $\lambda, \mu$ are fixed there are only finitely many choices for $\nu, \iota$. This is essentially Zygmund's observation \cite{zyg}.\\

In dimension $d=3$, if $2k=2$, we combine the same idea with a projection on an affine plane. More precisely, let $\lambda,\mu,\nu,\iota\in\Z^3$ be such that $p(\lambda)=p(\mu)=p(\nu)=p(\iota)=m$, $\lambda+\mu\neq 0$ and $\lambda+\mu+\nu+\iota=0$. Then, $\nu$ satisfies
\begin{equation}\label{eq:four_corr_prop_1}
p(\nu)=p(\lambda+\mu+\nu)=m\, .
\end{equation}
Let $M$ be a positive definite matrix with integer coefficients such that $\frac{1}{2}\langle x,Mx\rangle=p(x)$ for each $x\in\R^3$. Fix $\lambda$ and $\mu$ as above. Then, writing $w=\frac{1}{2}(\lambda+\mu)$, each solution $\nu$ of \eqref{eq:four_corr_prop_1} leads to a solution $x=\nu+w\in(1/2)\Z^3$ of
\[
\langle x-w,M(x-w)\rangle=\langle x+w,M(x+w)\rangle=2m
\]
which is in turn equivalent to
\begin{equation}\label{eq:four_corr_prop_2}
\langle x,Mw\rangle=0;\, \langle x,Mx\rangle =2m-\langle w,Mw\rangle\, .
\end{equation}
Without loss of generality, we may assume that the first coefficient of $Mw$ is a non-zero integer $h$ in the interval $[-C_1m^{1/2},C_1m^{1/2}]$ where $C_1=C_1(M)<+\infty$. In particular, multiplying the quadratic equation in \eqref{eq:four_corr_prop_2} by $h^2$ allows us to substitute both occurrences of $hx_1$ in the left-hand side by a linear combination of $x_2$ and $x_3$. We deduce that $\tilde{x}=(2x_2,2x_3)\in \Z^2$ satisfies a quadratic equation of the form
\[
\langle\tilde{x},Q\tilde{x}\rangle=m'
\]
where $Q$ has integer coefficients and $m'$ is an integer no greater than $C_2m^2$ where $C_2=C_2(M)<+\infty$. But by Lemma \ref{lemma:upper_bd_ellipses}, for each $\eps>0$ the number of such points $\tilde{x}$ is $O(m^\eps)$ uniformly in $Q$, that is, uniformly in the choice of $\lambda$ and $\mu$. Since, moreover, by Theorem \ref{thm:arithmetic_dimension_3}, $\calN_m\gg m^{1/2-\eps}$, we deduce that for each $\eps>0$, there exists $C_3=C_3(p)<+\infty$ such that for each $m\in\N$,
\[
|\calC_{4;m}|\leq C_3\calN_m^{2+\eps}\, .
\]
\end{proof}

\subsection{The high-dimensional case}\label{ss:arithmetic_hd}
The following proposition presents an equidistribution estimate in the case where $p$ is non-degenerate (see \eqref{eq:non_degenerate_polynomial}). Recall the measure $d\sigma_p$ from Definition \ref{def:equidistribution_measure}. The statement of Proposition \ref{prop:equidistribution_hd} involves the notion of \textit{Krull dimension} over $\R$ (see Chapter 1, section 1 of \cite{hartshorne}), which we will call \textit{algebraic dimension} throughout the paper for brevity.
\begin{prop}[\cite{magy02}, Theorem 1]\label{prop:equidistribution_hd}
Let $d'$ be the algebraic dimension of the variety $\{\nabla_xp=0\}$. Assume that
\[
d-d'>2^{2k}(2k-1)\, .
\]
Let $\mathfrak{S}\subseteq\Z$ be an infinite set of integers $m$ such that there exists $c>0$ for which $\calN_m\geq c m^{\frac{d}{2k}-1}$ whenever $m\in\mathfrak{S}$. Then, thre exist $C=C(p)<+\infty$ and $\delta=\delta(p)>0$ such that for any $\phi\in C^\infty_c(\R^d)$ and $m\in\mathfrak{S}$,
\[
\left|\frac{1}{\calN_m}\sum_{\lambda\in\Lambda_m}\phi(\lambda)-\int_{\Sigma_p} \phi(x)d\sigma_p(x)\right|\leq C\|\phi\|_\infty \calN_m^{-\delta}\, .
\]
\end{prop}
\begin{proof}
While the rate of decay is not specified in Theorem 1 of \cite{magy02}, it is explicitely given by its proof.
\end{proof}
To finish off the section, based on Theorem \ref{prop:birch_generic_systems} from \cite{birc62}, we prove Proposition \ref{prop:lattice_points_hd} and Theorem \ref{thm:correlations} which control the number of lattice points and $\ell$-correlations in the case where $p$ is non-degenerate (see \eqref{eq:non_degenerate_polynomial}).
\begin{thm}[\cite{birc62}, section 7, Theorem 1]\label{prop:birch_generic_systems}
Let $f_1,\dots,f_q$ be homogeneous polynomials of degree $2k$ with integer coefficients in $d_0$ variables. For each $m\in\N$, let $d_0'(m)$ be the algebraic dimension of the singular set of the variety
\begin{equation}\label{eq:diophantine_system}
f_1(x)=\, \dots=f_q(x)=m\, .
\end{equation}
There exist constants $c,C\in(0,+\infty)$ depending on $f_1,\dots,f_q$ such that the following holds. For each $m\in\N$ such that
\begin{equation}\label{eq:birch_generic_systems}
d_0-d_0'(m)>2^{2k-1}(2k-1)q(q+1)
\end{equation}
the number $\calN(m;f_1,\dots,f_q)$ of solutions $x\in\Z^{d_0}$ to \eqref{eq:diophantine_system} satisfies either $\calN(m;f_1,\dots,f_q)=0$ or
\[
c m^{\frac{d_0-2qk}{2k}}\leq \calN(m;f_1,\dots,f_q)\leq Cm^{\frac{d_0-2qk}{2k}}\, .
\]
\end{thm}
The case $q=1$ yields the following sharp bounds for the lattice points \cite{birc62,magy02,magybk}, in line with the heuristics \eqref{heu}.
\begin{prop}
\label{prop:lattice_points_hd}
Assume that $d>2^{2k}(2k-1)$.
Then, uniformly for each $m\in\N$, either $\calN_m=0$ or
\begin{equation}\label{eq:lattice_points_hd}
\calN_m\asymp m^{\frac{d}{2k}-1}\, .
\end{equation}
In particular, the set $\mathfrak{S}$ of $m\in\N$ for which $\calN_m>0$ has positive density.
\end{prop}
\begin{proof}
Since $p$ is elliptic, for any $m\in\N$, the set $p^{-1}(m)$ is smooth. Therefore Theorem \ref{prop:birch_generic_systems} applies directly and yields \eqref{eq:lattice_points_hd}. By
\[
T^{\frac{d}{2k}}\ll\left|\{x\in\Z^d\, :\, p(x)\leq T\}\right|\leq \sum_{0\leq m\leq T}\calN_m\overset{\eqref{eq:lattice_points_hd}}{\ll}  \left|\{m\in\Z\cap[0,T]\, :\, \calN_m>0\}\right|T^{\frac{d}{2k}-1}
\]
we deduce, as announced, that the set $\mathfrak{S}$ of $m\in\N$ for which $\calN_m>0$ has positive density. 
\end{proof}

\begin{thm}\label{thm:correlations}
Let $\overline{d}=d-\inf_{x\in\R^d\setminus\{0\}}\textup{rank}(\nabla^2 p(x))$. Assume that
\begin{equation}\label{eq:general_nondegeneracy_condition}
\min\{d,(\ell-2)(d-\overline{d})\}>\ell(\ell+1)\cdot 2^{2k-1}(2k-1)
\end{equation}
Then, uniformly for $m\in\N$, either $\calC_{\ell;m}=\emptyset$ or
\[
|\calC_{\ell;m}|\asymp m^{\frac{(\ell-1)d-2\ell k}{2k}}\asymp \calN_m^{(\ell-1)-\frac{2k}{d-2k}}\, .
\]
\end{thm}

\begin{rem}
In particular, the $\ell$-correlations of lattice points on ellipsoids satisfy
\begin{equation}
\label{correlli}
|\calC_{\ell;m}|\asymp\calN_m^{\ell-1-\frac{2}{d-2}}
\end{equation}
as soon as $d>2\ell(\ell+1)$. Indeed, here $2k=2$ and $\overline{d}=0$ (Remark \ref{rem:41.1}). The estimate \eqref{correlli} for ellipsoids is consistent with \eqref{corrsphu} and \eqref{corrsphl} for spheres. In the case of $4$-correlations, $d>2\ell(\ell+1)$ is just $d\geq 41$ (recall Remark \ref{rem:41.2}).
\end{rem}

To prove Theorem \ref{thm:correlations}, we invoke \cite[Theorem 1]{birc62} (see Theorem \ref{prop:birch_generic_systems}). Recall the definition \eqref{eq:degeneracy} of the degeneracy $\overline{d}$ of $p$.

\begin{proof}[Proof of Theorem \ref{thm:correlations}]
One considers the system
\begin{equation*}
\begin{cases}
p(\lambda_1)=\dots=p(\lambda_\ell)=m
\\
\lambda_1+\dots+\lambda_\ell=0,
\end{cases}
\end{equation*}
for $(\lambda_1,\dots,\lambda_\ell)\in\Z^{\ell d}$. This is equivalent to counting solutions to the system
\begin{equation}
\label{syst1}
p(x_1)=\dots=p(x_{\ell-1})=p(x_1+\dots+x_{\ell-1})=m, \qquad x,\dots,x_{\ell-1}\in\Z^{d}\, .
\end{equation}
Defining the new variable $w:=(x_1,\dots,x_{\ell-1})\in\Z^{(\ell-1)d}$, we rewrite \eqref{syst1} as
\begin{equation}
\label{syst2}
p_1(w)=\dots=p_\ell(w)=m\, .
\end{equation}
We are in a position to apply Theorem \ref{prop:birch_generic_systems} with parameters
\[
q=\ell, \, d_0=(\ell-1)d\, .
\]
Checking the inequality condition \eqref{eq:birch_generic_systems} requires computing the algebraic dimension $d_0'(m)$ of the subspace of $\R^{(\ell-1)d}$ where the $\ell\times (\ell-1)d$ Jacobian $J$ of \eqref{syst2} is not full-rank. The first $(\ell-1)$ rows of $J$ are already in row echelon form, with rank $(\ell-1)$ as long as each coordinate $x_j$ of $w$ is such that $\nabla p (x_j)\neq 0$. Since $p$ is elliptic homogeneous, its gradient vanishes only at zero and the dimension of the variety
\[
E_1=\cup_{j=1}^{\ell-1}\{\nabla p(x_j)=0\}
\]
is $(\ell-1)d-d=(\ell-2)d$.\\

On the complement of $E_1$, $J$ is of rank $\ell$ whenever its last row is not a linear combination of the first $\ell-1$ rows. Since $J_{\ell,j}=\dots=J_{\ell,(\ell-2)d+j}$ for $j=1,\dots,d$, to obtain a full-rank Jacobian we need to rule out
\begin{equation*}
E_2=\{\nabla p(x_1)=\dots=\nabla p(x_{\ell-1})\}\, .
\end{equation*}
The algebraic dimension of $E_2$ is at most equal to the dimension of the Zariski tangent space of $E_2$ at any point (see \cite{hartshorne} Exercise 5.10). To compute this dimension, we compute the rank of the differential of the following map
\[
(x_1,x_2,\dots,x_{\ell-1})\mapsto (\nabla p(x_1)-\nabla p(x_2),\dots,\nabla p(x_1)-\nabla p(x_{\ell-1}))\, .
\]
But the rank of the differential map is at least $(\ell-2)$ times the minimum of the rank of the Hessians $\nabla^2 p(w)$ for $w\in\R^d\setminus\{0\}$, which is $d-\overline{d}$. Hence, the dimension of the Zariski tangent space of $E_2$ at any point is at most $(\ell-1)d-(\ell-2)(d-\overline{d})=d+(\ell-2)\overline{d}$. All in all, $d_0'(m)\leq \max\{(\ell-2)d,d+(\ell-2)\overline{d}\}$. By Theorem \ref{prop:birch_generic_systems}, we conclude that, if
\begin{equation}
\label{cond1}
(\ell-1)d-\max\{(\ell-2)d,d+(\ell-2)\overline{d}\}=\min\{d,(\ell-2)(d-\overline{d})\}>\ell(\ell+1)\cdot 2^{2k-1}(2k-1),
\end{equation}
then, by Theorem \ref{prop:birch_generic_systems} there exist two constants $c=c(p)>0$, $C=C(p)<+\infty$ such that for each $m\in\N$, either $|\calC_{\ell;m}|=0$ or
\begin{equation}\label{eq:l_corr_estimate}
cm^{\frac{(\ell-1)d-2\ell k}{2k}}\leq|\calC_{\ell;m}|\leq Cm^{\frac{(\ell-1)d-2\ell k}{2k}}\, .
\end{equation}

Moreover, thanks to Proposition \ref{prop:lattice_points_hd}, as long as the bound
\begin{equation}
\label{cond2}
d>2^{2k}(2k-1),
\end{equation}
holds, then, for each $m\in\N$, either $\calN_m=0$ or
\begin{equation*}
\calN_m\asymp m^{\frac{d-2k}{2k}}\, .
\end{equation*}

Note that if $\calC_{\ell;m}$ is non empty, then $\calN_m>0$. Finally, note that \eqref{cond1} implies \eqref{cond2}. All in all, we conclude that whenever \eqref{cond1} holds, then, uniformly for $m\in\N$, either $\calC_{\ell;m}=\emptyset$ or
\begin{equation*}
|\calC_{\ell;m}|\asymp m^{\frac{(\ell-1)d-2\ell k}{2k}}\asymp \calN_m^{\frac{(\ell-1)d-2\ell k}{d-2k}}\, .
\end{equation*}
\end{proof}

\section{Main intermediate results and proof of Propositions \ref{prop:variance_upper_bound} and \ref{prop:varcur}, and of Theorems \ref{thm:variance} and \ref{thm:clt}}\label{s:main_proofs}

In the present section, we state two key intermediate results, namely Propositions \ref{prop:variance} and \ref{prop:second_chaos}, and use them in conjunction with the arithmetic results from section \ref{s:arithmetic} above to prove our main results.\\

Recall the definitions of $\Upsilon$ and $\Psi$ from \eqref{eq:upsilon_psi_definition}.

\begin{prop}[Arithmetic formula for the variance]\label{prop:variance}
Assume that $p$ satisfies Assumptions \ref{as:ample} and \ref{as:nd_eigvals}. For each $m\in\mathfrak{S}$ and for each $\lambda\in\Lambda_m$, let
\[
\Xi_m(\lambda)=\Upsilon(\Omega_m)^{-1}\langle\lambda,\Psi(\Omega_m)\lambda\rangle\, .
\]
Then, uniformly for $m\in\N$,
\[
\textup{Var}(\calV_m)=\frac{\pi}{\calN_m}\E[\calV_m]^2\frac{1}{\calN_m}\sum_{\lambda\in\Lambda_m}\br{1-\Xi_m(\lambda)}^2+O\br{\frac{L_m|\calC_{4;m}|}{\calN_m^4}}\, .
\]
The constant implied by the $O$ depends on the sequence $(\Omega_m)_m$.
\end{prop}
The proof of Proposition \ref{prop:variance} can be found in subsection \ref{ss:variance_conclusion}. It follows from the contents of section \ref{s:covariance_proofs}.

\begin{prop}[Arithmetic formula for the second chaotic projection]\label{prop:second_chaos}
For each $m\in\N$, with the notations of Proposition \ref{prop:variance},
\[
\calV_m[2]=\sqrt{\frac{L_m}{2\pi}}\Upsilon(\Omega_m)\frac{1}{\calN_m}\sum_{\lambda\in\Lambda_m}\left(\Xi_m(\lambda)-1\right)(|\zeta_\lambda|^2-1)\, .
\]
Moreover,
\[
\textup{Var}(\calV_m[2])=\frac{\pi}{\calN_m}\E[\calV_m]^2\frac{1}{\calN_m}\sum_{\lambda\in\Lambda_m}\left(1-\Xi_m(\lambda)\right)^2\, .
\]
\end{prop}
Proposition \ref{prop:second_chaos} is proved in section \ref{s:chaos}.\\

We now prove Propositions \ref{prop:variance_upper_bound} and \ref{prop:varcur}, and Theorems \ref{thm:variance} and \ref{thm:clt} using the two above propositions as well as the results from section \ref{s:arithmetic}.
\begin{proof}[Proof of Propositions \ref{prop:variance_upper_bound} and \ref{prop:varcur}, and of Theorems \ref{thm:variance} and \ref{thm:clt}]
Let us first assume that Assumptions \ref{as:ample} and \ref{as:nd_eigvals} are satisfied. Throughout the proof, we will use that, by Assumption \ref{as:nd_eigvals} and as noted below Definition \ref{def:omega}, the sequence $(\Omega_m)_{m\in\N}$ takes values in a compact set of the space of positive definite matrices so that the quantities $\Upsilon(\Omega_m)$ and the eigenvalues of the matrices $\Psi(\Omega_m)$ are bounded from above and below by positive constants uniformly in $m$.\\

By Proposition \ref{prop:expectation}, for $m$ large enough,
\begin{equation}\label{eq:main_proof_1}
\E[\calV_m]=\frac{\sqrt{L_m}}{2\pi}\Upsilon(\Omega_m)\asymp\sqrt{L_m}\, .
\end{equation}
Hence, by Proposition \ref{prop:variance},
\[
\frac{\textup{Var}(\calV_m)}{\E[\calV_m]^2}=O\left(\frac{1}{\calN_m}\max_{\lambda\in\Lambda_m}(1-\Xi_m(\lambda))^2\right)+O\left(\frac{|\calC_{4;m}|}{\calN_m^4}\right)\, .
\]
But by definition of $\Xi_m$, for each $\lambda\in\Lambda_m\subset\Sigma_p$,
\begin{equation}\label{eq:main_proof_5}
(1-\Xi_m(\lambda))^2\ll_p 1+\Xi_m(\lambda)^2\ll_p 1+|\lambda|^2\ll_p 1\, .
\end{equation}
Moreover, by construction, $|\calC_{4;m}|\leq \calN_m^3$. Hence,
\[
\frac{\textup{Var}(\calV_m)}{\E[\calV_m]^2}=O\left(\frac{1}{\calN_m}\right)\, .
\]
This concludes the proof of Proposition \ref{prop:variance_upper_bound}.\\

Next, we apply Propositions \ref{prop:variance} and \ref{prop:second_chaos} to deduce that
\begin{equation}\label{eq:main_proof_4}
\textup{Var}(\calV_m)=\frac{\pi}{\calN_m}\left(\E[\calV_m]\right)^2\frac{1}{\calN_m}\sum_{\lambda\in\Lambda_m}(1-\Xi_m(\lambda))^2+O\left(\frac{L_m|\calC_{4;m}|}{\calN_m^4}\right)=\textup{Var}(\calV_m[2])+O\left(\frac{L_m|\calC_{4;m}|}{\calN_m^4}\right)\, .
\end{equation}
An application of Proposition \ref{prop:four_correlations_dim_2_and_3} for $d=2$ now concludes the proof of Proposition \ref{prop:varcur}.

Moving on to the proof of Theorems \ref{thm:variance} and \ref{thm:clt}, let us assume from this point on that we are in case \textit{(A)} of Theorem \ref{thm:variance}. The proof can be easily adapted to the two other cases after a few changes which we present below.\\

Let $d\sigma_p$ be the measure from Definition \ref{def:equidistribution_measure}. Recall that $\Omega_m=\frac{1}{\calN_m}\sum_{\lambda\in\Lambda_m}\lambda\otimes\lambda^*$. Let $\mathfrak{S}\subset\N$ be the intersection of the sequences $\mathfrak{S}$ from Lemma \ref{lemma:lattice_points_2D} (say with $\eps=\frac{1}{4}$) and Proposition \ref{prop:equidistribution for ellipses}, which both apply since we are in case \textit{(A)}. In particular, Assumption \ref{as:ample} holds by Lemma \ref{lemma:lattice_points_2D}. We have, for each $\eps>0$, uniformly for $m\in\mathfrak{S}$, $\Omega_m=\Omega+O\left(\log(m)^{-\frac{1}{2}\log\left(\frac{\pi}{2}\right)+\eps}\right)$ where
\begin{equation}\label{eq:main_proof_2}
\Omega=\int_{\Sigma_p}x\otimes x^* d\sigma_p(x)
\end{equation}
Since, $d\sigma_p$ and the surface area measure on $\Sigma_p$ are mutually absolutely continuous, we deduce that $\Omega$ is positive definite and so Assumption \ref{as:nd_eigvals} also holds. Applying Proposition \ref{prop:equidistribution for ellipses} once again, uniformly for $m\in\mathfrak{S}$,
\begin{equation}\label{eq:main_proof_3}
\frac{1}{\calN_m}\sum_{\lambda\in\Lambda_m}(1-\Xi_m(\lambda))^2=\int_{\Sigma_p}(1-\Xi(x))^2d\sigma_p(x)+O\left(\log(m)^{-\frac{1}{2}\log\left(\frac{\pi}{2}\right)+\eps}\right)
\end{equation}
where
\[
\Xi(\lambda)=\Upsilon(\Omega)^{-1}\langle\lambda,\Psi(\Omega)\lambda\rangle\, .
\]
All in all, plugging this estimate into \eqref{eq:main_proof_4}, and using \eqref{eq:main_proof_1}, we deduce that, uniformly for $m\in\mathfrak{S}$,
\[
\frac{\textup{Var}(\calV_m)}{\E[\calN_m]^2}=\frac{\pi}{\calN_m}\times\int_{\Sigma_p}(1-\Xi(x))^2d\sigma_p(x)+O\left(\calN_m^{-1}\log(m)^{-\frac{1}{2}\log\left(\frac{\pi}{2}\right)+\eps}\right)+O\left(\frac{|\calC_{4;m}|}{\calN_m^4}\right)\, .
\]
By Proposition \ref{prop:four_correlations_dim_2_and_3}, we deduce that, uniformly for $m\in\mathfrak{S}$,
\[
\frac{\textup{Var}(\calV_m)}{\E[\calN_m]^2}=\frac{\pi}{\calN_m}\times\int_{\Sigma_p}(1-\Xi(x))^2d\sigma_p(x)+O\left(\calN_m^{-1}\log(m)^{-\frac{1}{2}\log\left(\frac{\pi}{2}\right)+\eps}\right)+O\left(\calN_m^{-2}\right)\, .
\]
This concludes the proof of case \textit{(A)} of Theorem \ref{thm:variance}. Still in this case, for Theorem \ref{thm:clt}, we use the second equality of \eqref{eq:main_proof_4} and Proposition \ref{prop:four_correlations_dim_2_and_3} once again to deduce that
\begin{equation}\label{eq:main_proof_6}
\frac{\textup{Var}(\calV_m)}{\E[\calN_m]^2}=\frac{\textup{Var}(\calV_m[2])}{\E[\calN_m]^2}+O\left(\calN_m^{-1}\log(m)^{-\frac{1}{2}\log\left(\frac{\pi}{2}\right)+\eps}\right)+O\left(\calN_m^{-2}\right)\, .
\end{equation}
But, by Proposition \ref{prop:second_chaos}, $\mathcal{V}_m[2]$ is a deterministic multiple of
\[
\sum_{\lambda\in\Lambda_m}c_m(\lambda)(|\zeta_\lambda|^2-1)
\]
where, $(\zeta_\lambda)_{\lambda}$ are complex standard normals which are mutually independent save for the relation $\zeta_{-\lambda}=\overline{\zeta}_\lambda$ and for each $\lambda$,
\[
c_m(\lambda):=\frac{\Xi_m(\lambda)-1}{\sqrt{\sum_{\lambda\in\Lambda_m}(\Xi_m(\lambda)-1)^2}}\, .
\]
The coefficients $c_m(\lambda)$ satisfy $\sum_\lambda c_m(\lambda)^2=1$ by construction. Moreover, by Proposition \ref{prop:equidistribution for ellipses} and \eqref{eq:main_proof_5}, their supremum is $O(\calN_m^{-1})$ which goes to zero by Lemma \ref{lemma:lattice_points_2D}. We conclude by applying Lindenberg principle (see for instance Remark 11.1.2 of \cite{noupec}) that $\frac{\calV_m[2]}{\sqrt{\textup{Var}(\calV_m[2])}}$ converges in law to a standard normal as $m\to+\infty$. By \eqref{eq:main_proof_6} and the last point of Proposition \ref{prop:upper_bd_ellipses}, this also holds for $\frac{\calV_m-\E[\calV_m]}{\sqrt{\textup{Var}(\calV_m)}}$.\\

This concludes the proof in case \textit{(A)} of Theorem \ref{thm:variance}. In case \textit{(B)} (resp. \textit{(C)}), the sequence $\mathfrak{S}$ is the one defined in Theorem \ref{thm:arithmetic_dimension_3} (resp. Proposition \ref{prop:lattice_points_hd}). Proposition \ref{prop:equidistribution for ellipses} is replaced by Theorem \ref{thm:arithmetic_dimension_3} (resp. Proposition \ref{prop:equidistribution_hd}) and the error term $O\left(\log(m)^{-\log\left(\frac{\pi}{2}\right)+\eps}\right)$ is thus replaced by $\calN_m^{-\frac{1}{111}+\eps}$ (resp. $\calN_m^{-\delta}$ for some $\delta=\delta(p)>0$). Here, for case \textit{(C)}, it should be noted that condition \eqref{eq:non_degenerate_polynomial} implies the assumption of Proposition \ref{prop:equidistribution_hd}, since $d'\leq \overline{d}$. Finally, Proposition \ref{prop:four_correlations_dim_2_and_3} applies in case \textit{(B)} yielding an error term $\calN_m^{-2+\eps}$ instead of $\calN_m^{-2}$ and should be replaced by Theorem \ref{thm:correlations} in case \textit{(C)}, which yields an error of order $\calN_m^{-\frac{d}{2k}}$. In order to ensure that the error terms go to zero, we need $\calN_m$ to go to infinity along $\mathfrak{S}$. This holds in case \textit{(B)} by the choice of $\mathfrak{S}$ from Theorem \ref{thm:arithmetic_dimension_3} and in case \textit{(C)} by Proposition \ref{prop:lattice_points_hd}. Moreover, these two results also yield the error estimates of Theorem \ref{thm:variance} in terms of powers of $m$ in cases \textit{(B)} and \textit{(C)}.

\end{proof}

\section{Expectation and variance asymptotics: proofs of Propositions \ref{prop:expectation} and \ref{prop:variance}}\label{s:covariance_proofs}

In this section, we prove Propositions \ref{prop:expectation} and \ref{prop:variance}. The proofs are structured as follows. In subsection \ref{ss:integral_expressions}, we prove Proposition \ref{prop:expectation} and establish an integral expression for the variance of the nodal set (see \eqref{eq:integral_for_second_moment}). In subsection \ref{ss:product_of_two_norms}, we derive a general expansion for the expectation of the product of the norms of two almost-independent components of a Gaussian vector. The goal is to apply this expansion to the integrand in \eqref{eq:integral_for_second_moment}. In subsections \ref{ss:singular_and_non_singular_sets} and \ref{ss:integrals_in_expansion}, we derive explicit expressions fo the integrals of the corresponding terms. We then conclude in subsection \ref{ss:variance_conclusion}. Finally, in subsections \ref{ss:variance_aux_1}, \ref{ss:contribution_of_the_singular_set} and \ref{ss:variance_aux_3}, we prove a series of auxiliary lemmas used in the previous subsections.

\subsection{Explicit expressions for expectation and variance}\label{ss:integral_expressions}
We begin with the proof of the expectation estimate Proposition \ref{prop:expectation}. In this proof we will employ an elegant method of Berry to deal with these and similar computations. In particular, we will apply the following identity. As observed in \cite[(24)]{berry2} (also see \cite[section 5]{krkuwi})
\begin{equation}
\label{eq:berry_s_method}
\sqrt{z}=\frac{1}{\sqrt{2\pi}}\int_{0}^{\infty}\left(1-\exp(-zt/2)\right)\frac{dt}{t^{3/2}}, \qquad z>0.
\end{equation}
Alternatively, one could proceed with a brute force Taylor expansion, but this calculation would be considerably longer. Recall that $r_m$ is the covariance matrix of $F_m$ and recall $L_m$ from \eqref{eq:l_m_def} as well as $\Omega_m$ and $\tilde{\Omega}_m$ from Definition \ref{def:omega}.
\begin{proof}[Proof of Proposition \ref{prop:expectation}]
Assume first that $\Omega_m$ is non-degenerate. Then, by stationarity, for each $x\in\T^d$, the random vector $(F_m(x),\nabla F_m(x))$ is non-degenerate. By the Kac-Rice formula (see Theorem 6.8 of \cite{azawsc}), using stationarity once again,
\begin{equation}
\label{be1}
\mathbb{E}[\mathcal{V}_m]=\frac{\sqrt{L_m}}{\sqrt{2\pi}}\cdot\mathbb{E}\left[\left|\frac{\nabla F_m(0)}{\sqrt{L_m}}\right|\right]\, .
\end{equation}
Applying \eqref{be1} with $z=|\frac{\nabla F_m(0)}{\sqrt{L_m}}|^2$ into \eqref{be1},
\begin{equation*}
\mathbb{E}[\mathcal{V}_m]=\frac{\sqrt{L_m}}{2\pi}\int_{0}^{\infty}\left(1-\mathbb{E}\left[\exp\left(-\frac{t}{2}\left|\frac{\nabla F}{\sqrt{L_m}}(0)\right|^2\right)\right]\right)\frac{dt}{t^{3/2}}.
\end{equation*}
By definition of expectation, for each $t$ we have
\begin{align*}
\mathbb{E}&\left[\exp\left(-\frac{t}{2}\left|\frac{\nabla F_m(0)}{\sqrt{L_m}}\right|^2\right)\right]=\frac{1}{\sqrt{(2\pi)^d\det(\Omega_m)}}\int_{\mathbb{R}^{d}}e^{-t|w|^2/2}
e^{-w^T\Omega_m^{-1}w/2}dw\\
&=(\det(\Omega_m)\det(tI_d+\Omega_m^{-1}))^{-1/2}=\det(I_d+t\Omega_m)^{-1/2}
\end{align*}
so that
\begin{equation*}
\mathbb{E}[\mathcal{V}_m]=\frac{\sqrt{L_m}}{2\pi}\int_{0}^{\infty}\left(1-\frac{1}{\sqrt{\det(t\Omega_m+I_d)}}\right)\frac{dt}{t^{3/2}}=\frac{\sqrt{L_m}}{2\pi}\Upsilon(\Omega_m)\, .
\end{equation*}
This completes the proof in the non-degenerate case. To extend it to the general case, note that $\Upsilon(\cdot)$ extends by continuity to symmetric matrices with non-negative eigenvalues. Thus, by adding a independent non-degenerate field multiplied by a small parameter $\eps$, we may perturb it in such a way that the covariance matrix of its gradient at any point converges to the original one. It remains to show that the volume of the perturbed field converges to the volume of the original field. We distinguish two cases. Either the gradient is a.s. zero and $\Omega_m=0$ in which case equality holds, or $\Omega_m$ is non-zero and the gradient is a.s. non-zero. By stationarity, we deduce that there exists a unit vector $v$ such that $(F_m(0),\langle\nabla F_m(0),v\rangle)$ is non-degenerate and so by Bulinskaya's lemma (see for instance Proposition 6.11 of \cite{azawsc}), $\nabla F_m$ does not vanish on the nodal set of $F_m$. Hence, as $\eps\to 0$, the length of the nodal set of the perturbation converges to that of $F_m$.
\end{proof}
Since $F_m$ is stationary, for each $x\in\T^d$, $x\neq 0$, the following expression does not depend on the choice of $y\in\T^d$:
\begin{equation}
\label{K2tdef}
\tilde{K}_{2;m}(x)
=\phi_{F_m(y),F_m(x+y)}(0,0)\cdot\mathbb{E}[|\nabla F_m(y)||\nabla F_m(x+y)|\ \big| \ F_m(y)=F_m(x+y)=0],
\end{equation}
Here, $\phi_{F_m(y),F_m(x+y)}$ denotes the density with respect to Lebesgue of the law of the random vector $(F_m(y),F_m(x+y))$. From the regression formula (see Proposition 1.2 of \cite{azawsc}) it is easy to see that $\tilde{K}_{2;m}(x)$ extends by continuity to $x=0$. By the Kac-Rice formula (see Theorem  6.9 of \cite{azawsc}), we have
\begin{equation}\label{eq:integral_for_second_moment}
\E\brb{\calV_m^2}=\int_{\T^d}\tilde{K}_{2;m}(x)dx\, .
\end{equation}
Thus, the study of variance asymptotics reduces to the study of $\tilde{K}_{2;m}$. It will be convenient to rescale $\tilde{K}_{2;m}$ by the parameter
\begin{equation}
L_m=4\pi^2m^{1/k}
\end{equation}
so we define
\begin{equation}
\label{eq:K2_def}
K_{2;m}(x):=\frac{\tilde{K}_{2;m}(x)}{L_m}\, .
\end{equation}
Let $\nabla r_m (x)$ be the column vector $(\partial_j r_m(x))_j$, and $\nabla^2 r_m(x)$ the Hessian matrix $(\partial_j\partial_l r_m(x))_{j,l}$.
\begin{prop}
\label{prop:second_intensity}
The scaled second intensity $K_{2;m}$ may be expressed in terms of the covariance function and its various first and second order derivatives as follows:
\begin{equation*}
K_{2;m}(x)=
\frac{1}{2\pi\sqrt{1-r_m^2(x)}}\mathbb{E}[|w_1||w_2|],
\end{equation*}
where $(w_1,w_2)\sim N(0,\Theta_m(x))$,
\begin{equation}
\label{eq:Theta_def}
\Theta_m(x)=\begin{pmatrix}
\Omega_m & 0 \\ 0 & \Omega_m
\end{pmatrix}
+
\begin{pmatrix}
U_m(x) & V_m(x) \\ V_m(x) & U_m(x)
\end{pmatrix},
\end{equation}
with
\begin{equation}
\label{eq:U_def}
U_m=-\frac{1}{L_m}\cdot\frac{1}{1-r_m^2}\nabla r_m\nabla r_m^T
\end{equation}
and
\begin{equation}
\label{eq:V_def}
V_m=-\frac{1}{L_m}
\left(
\nabla^2 r_m
+\frac{r_m}{1-r_m^2}
\nabla r_m\nabla r_m^T
\right).
\end{equation}
\end{prop}
\begin{proof}
Proceeding as in \cite[Proposition 4.2]{benmaf} and \cite[section 3]{krkuwi}, one writes the covariance matrix of the $2d+2$-dimensional Gaussian vector
\begin{equation*}
\left(
F_m(0),F_m(x),
\nabla F_m(0),
\nabla F_m(x)
\right)
\end{equation*}
as
$
\begin{pmatrix}
\Sigma_{11} & \Sigma_{12} \\ \Sigma_{21} & \Sigma_{22}
\end{pmatrix},
$ 
with
\begin{equation*}
\Sigma_{11}=
\begin{pmatrix}
1 & r_m \\r_m & 1
\end{pmatrix},
\qquad
\Sigma_{12}=\Sigma_{21}=
\begin{pmatrix}
0 & \nabla r_m(x)\\-\nabla r_m(x) & 0
\end{pmatrix},
\qquad
\Sigma_{22}=
\begin{pmatrix}
\tilde{\Omega}_m & -\nabla^2 r_m(x)
\\
-\nabla^2 r_m(x) & \tilde{\Omega}_m
\end{pmatrix}.
\end{equation*}
Therefore,
\begin{equation*}
\tilde{K}_{2;m}(x)=\frac{1}{2\pi\sqrt{1-r_m^2(x)}}\cdot\mathbb{E}[|v_1||v_2|],
\qquad
(v_1,v_2)\sim N(0,\tilde{\Theta}),
\end{equation*}
with $\tilde{\Theta}=\Sigma_{22}-\Sigma_{21}\Sigma_{11}^{-1}\Sigma_{12}$. Rescaling by $L_m$ completes the proof.
\end{proof}
\subsection{Expectation of product of two norms}\label{ss:product_of_two_norms}
The main tool we use to compute \eqref{eq:integral_for_second_moment} is the following expansion for the expectation of the product of the norms of two almost independent Gaussian vectors. The case where $\Omega=I_d$ has been used in \cite{berry2,krkuwi}. To the best of our knowledge, this is the first time it has been generalised to the non-isotropic case (i.e. for any $\Omega$).
\begin{lemma}
\label{lemma:berry_s_method}
Let $\Omega$ be a $d\times d$ positive definite symmetric matrix. Let $\Theta$ be a $2d\times 2d$ positive definite symmetric matrix of the form
\[
\Theta=\br{\begin{matrix}
\Omega & 0\\
0 & \Omega
\end{matrix}}+\br{\begin{matrix}
U & V\\
V & U
\end{matrix}}\, .
\]
Let $(v_1,v_2)\sim N(0,\Theta)$. Assume that there exists $C_1<+\infty$ such that $\|U\|_\infty\leq C_1$ and $\|V\|_\infty\leq C_1$. Let
\begin{equation}
\Upsilon(\Omega):=\int_{0}^{\infty}\left(1-\frac{1}{\sqrt{\det(I_d+t\Omega)}}\right)\frac{dt}{t^{3/2}}\text{ and }\Psi(\Omega):=\int_0^\infty t^{-1/2}\det(I_d+t\Omega)^{-1/2}(I_d+t\Omega)^{-1}dt\, .
\end{equation}
Then,
\begin{equation*}
\mathbb{E}[|v_1||v_2|]
=
\Upsilon(\Omega)^2+
\Upsilon(\Omega)\tr(U\Psi(\Omega))+
\frac{1}{2}\tr(V\Psi(\Omega)V\Psi(\Omega))
+
O(\|U\|_\infty^2+\|V\|_\infty^4)
\end{equation*}
where the constant implied by $O(\cdot)$ may depend on $\Omega$ and $C_1$.
\end{lemma}

\subsection{The contributions of the singular and non-singular sets: setup}\label{ss:singular_and_non_singular_sets}
One major difficulty when working with the random field $F_m$ is that its covariance function, determined by $r_m(x)=\E[F_m(0)F_m(x)]$ does not decay at large distances. To deal with this difficulty, we follow the same strategy as in \cite{orruwi,rudwi2,krkuwi,benmaf}. The strategy is, we define a small ``singular set'' $S$, then compute the asymptotic of $K_2$ outside of $S$, and bound the contribution of $K_2$ on $S$.\\

The polynomial $p$ is elliptic so
\begin{equation}\label{eq:c_p_def}
c_p:=\sup_{m\in\mathfrak{S}}\sup_{\lambda\in\Lambda_m}|\lambda|<+\infty\, .
\end{equation}
Moreover, by Assumption \ref{as:nd_eigvals}, the eigenvalues $(a_j(m))_{j=1,\dots,d}$ of $\Omega_m$ are all bounded from below 
uniformly in $m\in\mathfrak{S}$ by a positive constant $a_p$. Note that for each $m\in\mathfrak{S}$,
\[
da_p\leq \tr(\Omega_m)=\frac{1}{\calN_m}\sum_{\lambda\in\Lambda_m}|\lambda|^2\leq c_p^2
\]
\begin{equation}\label{eq:non_sing_bound}
0<\frac{a_p}{c_p^2}\leq \frac{1}{d}\leq \frac{1}{2}\, .
\end{equation}

\begin{defin}
\label{def:singular_set}
We call the point $x\in\mathbb{T}^d$ \textit{positive singular} (resp. \textit{negative singular}) if there exists a subset $\Lambda_x\subseteq\Lambda$ of density
\begin{equation*}
\frac{|\Lambda_x|}{|\Lambda_m|}>1-\frac{a_p}{4c_p^2}
\end{equation*}
such that $\cos 2\pi m^{1/2k}\langle\lambda,x\rangle>\frac{3}{4}$ (resp. $\cos 2\pi m^{1/2k}\langle\lambda,x\rangle<-\frac{3}{4}$) for all $\lambda\in\Lambda_x$. We cover $\T^d$ with a family $\calQ_m$ of $d$-cubes with disjoint interiors and side length $q_m\geq 8\pi\sqrt{d}c_pm^{1/2k}$. In particular, $|\calQ_m|\asymp q_m^d$. The \textit{singular set} $S_m$ is the union of all the cubes in $\calQ_m$ containing a (positive or negative) singular point.
\end{defin}
Firstly, the quantity $r_m(x)$ is small for $x\notin S_m$.
\begin{lemma}\label{lemma:off_singular_bound}
Let $m\in\mathfrak{S}$. Then, for all $x\in \T^d\setminus S_m$, $|r_m(x)|<1-\frac{1}{32}$.
\end{lemma}
\begin{proof}
Reasoning as in \cite[Lemma 6.5 (i)]{orruwi} we deduce that for $x\notin S_m$, $|r_m(x)|<1-\frac{a_p}{16 c_p^2}$ and we conclude by \eqref{eq:non_sing_bound}.
\end{proof}
Secondly, the contribution of $S_m$ to the integral defining the second moment \eqref{eq:integral_for_second_moment} is small.
\begin{lemma}\label{lemma:contribution_of_the_singular_set}
For each $l\in\N$, there exists $C=C(p,l)<+\infty$ such that for each $m\in\mathfrak{S}$, the following holds:
\begin{itemize}
\item The volume of the singular set satisfies the bound:
\[
|S_m|\leq C\int_{\T^d} |r(x)|^l dx\, .
\]
\item The integral of the two-point intensity over the singular set satisfies:
\[
\int_{S_m}|K_{2;m}(x)|dx\leq C\int_{\T^d}|r(x)|^ldx\, .
\]
\end{itemize}
\end{lemma}
The proof of Lemma \ref{lemma:contribution_of_the_singular_set} is a variation on \cite[section 6]{orruwi}. We postpone it until subsection \ref{ss:contribution_of_the_singular_set}.

\subsection{Estimating integrals in the $K_{2;m}$ expansion}\label{ss:integrals_in_expansion}

Recall the definitions on $U_m$ and $V_m$ given in \eqref{eq:U_def} and \eqref{eq:V_def} respectively. When computing \eqref{eq:integral_for_second_moment}, we encounter a sum of principal terms and three integral remainder terms which we estimate here. The common upper bound is expressed in arithmetic terms as follows. For each $m\in\mathfrak{S}$, recall that $\calN_m=|\Lambda_m|$ and that $\calC_{4;m}=\{(\lambda,\mu,\nu,\iota)\in\Lambda_m^4\, :\, \lambda+\mu+\nu+\iota=0\}$ is the set of four-correlations in $\Lambda_m$.

\begin{lemma}[$\Theta_m(x)$, $U_m(x)$ and $V_m(x)$ are bounded]\label{lem:U_V_bound}
The matrices $\Theta_m(x)$, $U_m(x)$ and $V_m(x)$ are uniformly bounded in $m\in\mathfrak{S}$ and $x\in\T^d$.
\end{lemma}
The proof of this lemma is identical to that of Lemma 5.5 \cite{benmaf} and Lemma 3.2 of \cite{krkuwi}
\begin{proof}
The matrix $\Theta_m(x)$ from Proposition \ref{prop:second_intensity} is a covariance matrix so its individual entries are bounded by its diagonal entries. Now, the diagonal coefficients of the matrices $\Theta_m(x)$ are given by the variances of $L_m^{-1/2}\partial_jF_m(0)$ conditionitioned on $F_m(0)=F_m(x)=0$. But by the regression formula (see Proposition 1.2 of \cite{azawsc}), conditioning only reduces the variance, and the unconditioned variance of $L_m^{-1/2}\partial_j F_m(0)$ is exactly
\[
\frac{1}{\calN_m}\sum_{\lambda\in\Lambda_m}\lambda_j^2\leq \max\{|\lambda|^2\, :\, \lambda \in \cup_m\Lambda_m\}<+\infty\, .
\]
These unconditioned variances are also the diagonal coefficients of $\Omega_m$. All in all, the entries of $\Theta_m(x)$ and $\Omega_m$ are uniformly bounded in $x$ and $m$ so the same is true for $U_m(x)$ and $V_m(x)$.
\end{proof}
Using Lemma \ref{lem:U_V_bound}, we can estimate the various integral terms appearing in the expansion stemming from Lemma \ref{lemma:berry_s_method} and \eqref{eq:integral_for_second_moment}. We gather them in the present lemma whose proof we postpone until subsection \ref{ss:variance_aux_1}.
\begin{lemma}\label{lemma:remainder_estimates}
Uniformly for each $m\in\mathfrak{S}$, the following holds:
\begin{align*}
\int_{\T^d}r_m(x)^2dx&=\frac{1}{\calN_m}\, ;\\
\int_{\T^d} r_m(x)^4dx&=\frac{|\calC_{4;m}|}{\calN_m^4}\, ;\\
\int_{\T^d}\tr(U_m(x)\Psi(\Omega_m))dx&=-\frac{1}{\calN_m^2}\sum_{\lambda\in\Lambda_m}\langle\lambda,\Psi(\Omega_m)\lambda\rangle+O_p\left(\frac{|\calC_{4;m}|}{\calN_m^4}\right)+O_p(|S_m|)\, ;\\
\int_{\T^d}\tr\br{V_m(x)\Psi(\Omega_m)V_m(x)\Psi(\Omega_m)}dx&=\frac{1}{\calN_m^2}\sum_{\lambda\in\Lambda_m}\langle\lambda,\Psi(\Omega_m)\lambda\rangle^2+O_p\left(\frac{|\calC_{4;m}|}{\calN_m^4}\right)+O_p(|S_m|)\, ;\\
\int_{\T^d}\|U_m(x)\|_\infty^2+\|V_m(x)\|_\infty^4dx&\ll_p \frac{|\calC_{4;m}|}{\calN_m^4}+|S_m|\, .
\end{align*}
\end{lemma}

\subsection{Conclusion: proof of Proposition \ref{prop:variance} }\label{ss:variance_conclusion}

We now derive Proposition \ref{prop:variance} from the results presented in the previous subsections of section \ref{s:covariance_proofs}.
\begin{proof}[Proof of Proposition \ref{prop:variance}]
Let $m\in\mathfrak{S}$. By equations \eqref{eq:integral_for_second_moment} and \eqref{eq:K2_def}
\[
\E[\calV_m^2]=L_m\int_{\T^d} K_{2;m}(x)dx\, .
\]
Let $(w_1,w_2)$, $\Theta_m(x)$, $U_m(x)$ and $V_m(x)$ be as in Proposition \ref{prop:second_intensity}. By Proposition \ref{prop:second_intensity}, for each $x\in\T^d\setminus S_m$,
\begin{align*}
2\pi K_{2;m}(x)&\overset{\textup{Proposition }\ref{prop:second_intensity}}{=}\frac{1}{\sqrt{1-r_m(x)^2}}\E[|w_1||w_2|]\\
&\overset{\textup{Lemma }\ref{lemma:off_singular_bound}}{=}\left(1+\frac{1}{2}r_m(x)^2\right)\E[|w_1||w_2|]+O(r_m(x)^4\E[w_1^2])\\
&\overset{\textup{Lemma }\ref{lem:U_V_bound}}{=}\left(1+\frac{1}{2}r_m(x)^2\right)\E[|w_1||w_2|]+O(r_m(x)^4)\\
&\overset{\textup{Lemma }\ref{lemma:berry_s_method}}{=}\left(1+\frac{1}{2}r_m(x)^2\right)\times\\
&\left\{\Upsilon(\Omega_m)^2+\Upsilon(\Omega_m)\tr(U_m(x)\Psi(\Omega_m))+\frac{1}{2}\tr(V_m(x)\Psi(\Omega_m)V_m\Psi(\Omega_m)\right\}\\
&+O\left(\|U_m(x)\|_\infty^2+\|V_m(x)\|_\infty^4+r_m(x)^4\right)\\
&=\br{1+\frac{1}{2}r_m(x)^2}\Upsilon(\Omega_m)^2+\Upsilon(\Omega_m)\tr(U_m(x)\Psi(\Omega_m))+\frac{1}{2}\tr(V_m(x)\Psi(\Omega_m)V_m\Psi(\Omega_m))\\
&+O\left(r_m(x)^4+\|U_m(x)\|_\infty^2+\|V_m(x)\|_\infty^4\right)\, .
\end{align*}
In the last line we used that $\Upsilon(\Omega_m)$ and $\Psi(\Omega_m)$ are uniformly bounded so $r_m(x)^2=O(r_m(x)^2\|U_m(x)\|_\infty)=O(r_m(x)^4+\|U_m(x)\|_\infty^4)$ and similarly for the term involving $V_m(x)$. By integrating over $\T^d\setminus S_m$ and $S_m$ and controlling the error terms coming from the integrals over $S_m$ using Lemmas \ref{lem:U_V_bound}, Lemma \ref{lemma:contribution_of_the_singular_set} and \ref{lemma:remainder_estimates}, we deduce that
\begin{multline*}
2\pi\int_{\T^d}K_{2;m}(x)dx=\br{1+\frac{1}{2}\times\frac{1}{\calN_m}}\Upsilon(\Omega_m)^2\\-\Upsilon(\Omega_m)\times\frac{1}{\calN_m^2}\sum_{\lambda\in\Lambda_m}\langle\lambda,\Psi(\Omega_m)\lambda\rangle+\frac{1}{2}\times\frac{1}{\calN_m^2}\sum_{\lambda\in\Lambda_m}\langle\lambda,\Psi(\Omega_m)\lambda\rangle^2+O\br{\frac{|\calC_{4;m}|}{\calN_m^4}}\, .
\end{multline*}
In the above computation the remainder $O(|S_m|)$ is bounded by $\int_{\T^d}r_m(x)^4dx$ using Lemma \ref{lemma:contribution_of_the_singular_set} which is bounded by $\calN_m^{-4}\calC_{4;m}$ by Lemma \ref{lemma:remainder_estimates}. By Proposition \ref{prop:expectation}, writing $\Xi_m(\lambda)=\Upsilon(\Omega_m)^{-1}\langle\lambda,\Psi(\Omega_m)\lambda\rangle$,
\begin{align*}
2\pi\int_{\T^d}K_{2;m}(x)dx&=\frac{4\pi^2}{L_m}\E[\calV_m]^2+\frac{\Upsilon(\Omega_m)^2}{2\calN_m}\times\frac{1}{\calN_m}\sum_{\lambda\in\Lambda_m}\br{1-2\Xi_m(\lambda)+\Xi_m(\lambda)^2}+O\br{\frac{|\calC_{4;m}|}{\calN_m^4}}\\
&=\frac{4\pi^2}{L_m}\E[\calV_m]^2+\frac{1}{2\calN_m}\times \frac{4\pi^2}{L_m}\E[\calV_m]^2\times\frac{1}{\calN_m}\sum_{\lambda\in\Lambda_m}\br{1-\Xi_m(\lambda)}^2+O\br{\frac{|\calC_{4;m}|}{\calN_m^4}}\, .
\end{align*}
\end{proof}
Hence,
\[
\textup{Var}(\calV_m)=\frac{\pi}{\calN_m}\times\E[\calV_m]^2\times\frac{1}{\calN_m}\sum_{\lambda\in\Lambda_m}\br{1-\Xi_m(\lambda)}^2+O\br{\frac{L_m|\calC_{4;m}|}{\calN_m^4}}\, .
\]
\subsection{Proof of Lemma \ref{lemma:remainder_estimates}}\label{ss:variance_aux_1}

\begin{proof}[Proof of Lemma \ref{lemma:remainder_estimates}]
\begin{itemize}
\item We prove the first point as follows: $\int_{\T^d}r_m(x)^2dx=\frac{1}{\calN_m}\times\frac{1}{\calN_m}\sum_{\lambda,\lambda'\in\Lambda_m}\int_{\T^d}e(\langle\lambda-\lambda',x\rangle)dx=\frac{1}{\calN_m}\sum_{\lambda,\lambda'\in\Lambda_m}\delta_{\lambda=\lambda'}=\frac{1}{\calN_m}$.
\item The second point is similar:
\[
\int_{\T^d} r_m(x)^4dx=\frac{1}{\calN_m^4}\sum_{\lambda_1,\dots,\lambda_4\in\Lambda_m}\delta_{\lambda_1+\cdots+\lambda_4=0}=\frac{|\calC_{4;m}|}{\calN_m^4}\, .
\]
\item To estimate the integral of $\tr(U_m(x)\Psi(\Omega_m))$ with $U_m(x)=-\frac{1}{L_m}\frac{1}{1-r_m(x)^2}\nabla r_m(x)\nabla r_m(x)^T$ recall that by Lemma \ref{lem:U_V_bound}, $U_m(x)$ is uniformly bounded on the singular set $S_m$ and observe that by the Cauchy-Schwarz inequality and the definition of $r_m$, $(L_m^{-1/2}\partial _jr_m(x))^2\leq r_m(0)(\Omega_m)_{jj}$ which is uniformly bounded in $j$, $m\in\mathfrak{S}$ and $x\in\T^d$. On the other hand, by Lemma \ref{lemma:off_singular_bound}, on $\T^d\setminus S_m$, $(1-r_m(x)^2)^{-1}=1+O(r_m^2)$. Hence,
\begin{align*}
\int_{\T^d} U_m(x)dx&=\int_{\T^d\setminus S_m} U_m(x)dx+O_p(|S_m|)\\
&=\int_{\T^d\setminus S_m} -\frac{1}{L_m}\nabla r_m(x)\nabla r_m(x)^Tdx+O(r_m(x)^2|\nabla r_m(x)|^2)+O_p(|S_m|)\\
&=\int_{\T^d} -\frac{1}{L_m}\nabla r_m(x)\nabla r_m(x)^Tdx+O(r_m(x)^2|\nabla r_m(x)|^2)+O_p(|S_m|)\\
&=-\frac{1}{\calN_m^2}\sum_{\lambda\in\Lambda_m}\lambda\lambda^T+\int_{\T^d}O(L_m^{-1}r_m(x)^2|\nabla r_m(x)|^2)dx+O_p(|S_m|)\, .
\end{align*}
But $r_m(x)^2(L_m^{-1/2}\partial_jr_m(x))^2=\calN_m^{-4}\sum_{\lambda_1,\dots,\lambda_4\in\Lambda}(\lambda_3)_j(\lambda_4)_je(\langle\lambda_1+\cdots+\lambda_4,x\rangle)$ so
\[
\int_{\T^d}O(L_m^{-1}r_m(x)^2|\nabla r_m(x)|^2)dx=O_p\left(\frac{|\calC_{4;m}|}{\calN_m^4}\right)\, .
\]
All in all, since $\Psi(\Omega_m)$ is uniformly bounded (see below Definition \ref{def:omega} and below \eqref{eq:upsilon_psi_definition}), as announced,
\[
\int_{\T^d} \tr(U_m(x)\Psi(\Omega_m))dx=-\frac{1}{\calN_m^2}\sum_{\lambda\in\Lambda_m}\langle\lambda,\Psi(\Omega_m)\lambda\rangle+O_p\left(\frac{|\calC_{4;m}|}{\calN_m^4}\right)+O_p(|S_m|)\, .
\]
\item Reasoning as for $U_m(x)$, we deduce that the integral of $\tr\br{V_m(x)\Psi(\Omega_m)V_m(x)\Psi(\Omega_m)}$ where $V_m(x)=-\frac{1}{L_m}(\nabla^2r_m(x)+\frac{r_m(x)}{1-r_m(x)^2}\nabla r_m(x)\nabla r_m(x)^T)$ satisfies the following expansion:
\begin{multline*}
\int_{\T^d}\tr\br{V_m(x)\Psi(\Omega_m)V_m(x)\Psi(\Omega_m)}dx=\frac{1}{L_m^2}\int_{\T^d}\tr\br{\nabla^2r_m(x)\Psi(\Omega_m)\nabla^2r_m(x)\Psi(\Omega_m)}dx\\
+\int_{\T^d}O(L_m^{-2}r_m(x)\tr(\nabla^2r_m(x)\Omega_m\nabla r_m(x)\nabla r_m^T(x))dx\\
+\int_{\T^d}O(L_m^{-2}r_m(x)^2\tr(\nabla r_m(x)\nabla r_m(x)^T\Omega_m\nabla r_m(x)\nabla r_m(x)^T\Omega_m)dx+O_p(|S_m|)\, .
\end{multline*}
But just as for $U_m(x)$ the two first remainder terms may be bounded by $\frac{|\calC_{4;m}|}{\calN_m^4}$ so
\begin{multline*}
\int_{\T^d}\tr\br{V_m(x)\Psi(\Omega_m)V_m(x)\Psi(\Omega_m)}dx=\frac{1}{L_m^2}\int_{\T^d}\tr\br{\nabla^2r_m(x)\Psi(\Omega_m)\nabla^2r_m(x)\Psi(\Omega_m)}dx\\
+O_p\left(\frac{|\calC_{4;m}|}{\calN_m^4}\right)+O_p(|S_m|)\, .
\end{multline*}
To conclude, using, as above, the orthogonality properties of the maps $x\mapsto e(\langle\xi,x\rangle)$ for different $\xi$ and the cyclicity of the trace,
\[
\frac{1}{L_m^2}\int_{\T^d}\tr\br{\nabla^2r_m(x)\Psi(\Omega_m)\nabla^2r_m(x)\Psi(\Omega_m)}dx=\frac{1}{\calN_m^2}\sum_{\lambda\in\Lambda_m}\langle\lambda,\Psi(\Omega_m)\lambda\rangle^2\, .
\]
\item To upper bound the integral of $\|U_m(x)\|_\infty^2+\|V_m(x)\|^4_\infty$, we reason as above. First, by removing the singular set, we can bound $(1-r_m(x)^2)^{-1}$ uniformly, both in $U_m(x)$ and $V_m(x)$. Second, we note that each term is a polynomial in $r_m(x)$ and its derivatives whose monomials are of degree at least four. Moreover, the total number of derivatives in the factors of each monomial is equal to the power of $L_m^{-1/2}$ appearing in front of it. Therefore, we have, uniformly for $m\in\mathfrak{S}$ and $x\in \T^d\setminus S_m$,
\[
\|U_m(x)\|_\infty^2+\|V_m(x)\|^4_\infty\ll\frac{1}{\calN_m^4}\sum_{\lambda_1,\dots,\lambda_4\in\Lambda_m}P(\lambda_1,\lambda_2,\lambda_3,\lambda_4)e(\langle\lambda_1+\cdot+\lambda_4,x\rangle)
\]
where $P$ is a univeral polynomial in four variables. The right-hand side is uniformly bounded in $m$ and $x$ and its integral on $\T^d$ is bounded by $\frac{|\calC_{4;m}|}{\calN_m^4}$ so
\[
\int_{\T^d}\|U_m(x)\|_\infty^2+\|V_m(x)\|^4_\infty dx\ll_p \frac{|\calC_{4;m}|}{\calN_m^4}+|S_m|\, .
\]
\end{itemize}
\end{proof}

\subsection{Proof of Lemma \ref{lemma:contribution_of_the_singular_set}}\label{ss:contribution_of_the_singular_set}
\begin{proof}[Proof of Lemma \ref{lemma:contribution_of_the_singular_set}]
We begin by proving that for each $x\in S_m$,
\begin{equation}\label{eq:singular_set_proof_1}
|r_m(x)|>\frac{1}{4}\, .
\end{equation}
Indeed, let $Q\in\calQ_m$ containing a positive singular point $x$. By definition of $q_m$ and $c_p$ (see Definition \ref{def:singular_set} and \eqref{eq:c_p_def}), for each $y\in Q$, $\cos2\pi m^{1/2k}\langle \lambda,y\rangle\geq \frac{3}{4}-2\pi|\lambda|m^{1/2k}/q_m\geq \frac{1}{2}$.  Therefore,
\[
r_m(x)=\frac{1}{\calN_m}\sum_{\lambda\in\Lambda_m}\cos(2\pi m^{1/2k}\langle\lambda,x\rangle)>\frac{1}{|\Lambda_m|}\left(\frac{1}{2}|\Lambda_x|-|\Lambda_m\setminus\Lambda_x|\right)>\frac{1}{2}-\frac{3a_p}{8c_p^2}\overset{\eqref{eq:non_sing_bound}}{\geq} \frac{1}{2}-\frac{3}{16}>\frac{1}{4}\, .
\]
Similarly, if $x$ is negative singular, $r_m(x)<-\frac{1}{4}$ so \eqref{eq:singular_set_proof_1} holds. From \eqref{eq:singular_set_proof_1} and the Markov inequality, we deduce that for each $l\in\N$,
\begin{equation}\label{eq:singular_set_proof_2}
\textup{Vol}(S_m)<4^l\int_{\T^d}|r(x)|^ldx
\end{equation}
This proves the first point of the lemma. To prove the second point, it now suffices to show that
\begin{equation}\label{eq:singular_set_proof_3}
\int_{S_m}K_{2;m}(x)dx\leq C\textup{Vol}(S_m)
\end{equation}
for some constant $C=C(p)<+\infty$ independent of $m\in\mathfrak{S}$. To do so, we first observe that by Proposition \ref{prop:second_intensity}, for each $x\in\T^d$,
\[
K_{2;m}(x)=\frac{1}{2\pi\sqrt{1-r_m(x)^2}}\E[|w_1||w_2|]
\]
where the pair $(w_1,w_2)$ has the law of the pair $(\nabla F_m(0),\nabla F_m(x))$ conditioned on $F_m(0)=F_m(x)=0$. By Cauchy-Schwarz and stationarity, since variances of Gaussian vectors can only decrease under linear conditioning\footnote{This is a consequence of the regression formula. See Proposition 1.2 of \cite{azawsc}.}, we obtain
\begin{equation}\label{eq:singular_set_proof_4}
K_{2;m}(x)\leq \frac{\E[L_m^{-1}|\nabla F_m(0)|^2]}{2\pi\sqrt{1-r_m(x)^2}}=\frac{\tr(\Omega_m)}{2\pi\sqrt{1-r_m(x)^2}}\overset{\eqref{eq:c_p_def}}{\leq}\frac{c_p^2}{2\pi}\times\frac{1}{\sqrt{1-r_m(x)^2}}\, .
\end{equation}
If $x$ is positive singular, contained in some cube $Q\in\calQ_m$, we have no hope of finding a uniform upper bound for $(1-r_m^2)^{-1/2}$ on $Q$. Instead, we look for an integrable upper bound. For each $y\in Q$, let $\nabla^2r_m(y)$ be the Hessian of $r_m$ at $y$. We claim that for each $v\in\R^d$, and $y\in Q$,
\begin{equation}\label{eq:singular_set_proof_5}
\langle v,\nabla^2r_m(y) v\rangle\leq -c m^{1/k}|v|^2
\end{equation}
for some $c=c(p)>0$ independent of $m$, $x$ or $v$. Indeed, for each $y\in Q$ and $\lambda\in \Lambda_Q$, $\cos 2\pi m^{1/2k}\langle \lambda,y\rangle \geq \frac{1}{2}$ so that for each $v\R^d$,
\[
\langle v,H_\lambda(y)v\rangle:=\langle v,\nabla^2\cos (2\pi m^{1/2k}\langle \lambda,y\rangle)v\rangle=-4\pi^2m^{1/k}\cos 2\pi m^{1/2k}\langle \lambda,y\rangle\langle\lambda,v\rangle^2\leq -2\pi^2m^{1/k}\langle\lambda,v\rangle^2.
\]
On the other hand, if $\lambda\in\Lambda_m\setminus\Lambda_Q$, $\langle v,H_\lambda(y)v\rangle\leq 4\pi^2m^{1/k}\langle\lambda,v\rangle^2$ so that
\begin{multline*}
\langle v,\nabla^2 r_m(y)v\rangle\leq \frac{4\pi^2m^{1/k}}{\calN_m}\br{-1/2\sum_{\lambda\in\Lambda_Q}\langle\lambda,v\rangle^2+\sum_{\lambda\in\Lambda_m\setminus\Lambda_Q}\langle\lambda,v\rangle^2 }\\
\overset{\eqref{eq:non_sing_bound}}{\leq} 4\pi^2m^{1/k}\left(-1/2\langle v,\Omega_mv\rangle+\frac{a_p}{8c_p^2}c_p^2|v|^2\right)\leq 4\pi^2m^{1/k}(-7/8)a_p|v|^2
\end{multline*}
which proves \eqref{eq:singular_set_proof_5}. As in \cite[section 6.5]{orruwi}, we assume that $x$ is the maximum of $r_m$ on $Q$ and deduce, from Taylor expansion to order two, that for each $y\in Q$, $r_m(y)\leq 1-(c/2)m^{1/k}|x-y|^2$, from which we deduce that
\[
\int_Q\frac{dy}{\sqrt{1-r_m(y)^2}}\ll_p\textup{Vol}(Q)\, .
\]
Here we use that, by Definition \ref{def:singular_set}, the side length of $Q$ is $q_m=O(m^{-1/2k})$. By symmetry, the same estimate holds for cubes $Q$ containing a negative singular point so we deduce that
\[
\int_{S_m}\frac{dy}{\sqrt{1-r_m(y)^2}}\ll_p\textup{Vol}(S_m)\, .
\]
Together with \eqref{eq:singular_set_proof_4}, this proves \eqref{eq:singular_set_proof_3}, which, together with \eqref{eq:singular_set_proof_2}, completes the proof of the lemma.
\end{proof}

\subsection{Proof of Lemma \ref{lemma:berry_s_method}}\label{ss:variance_aux_3}
\begin{proof}[Proof of Lemma \ref{lemma:berry_s_method}.]
We adapt and generalise \cite[Lemma 5.1]{krkuwi} and \cite[Lemma 5.8]{benmaf} to our case. Here we present the main steps, highlighting the key differences with the cited papers, and we refer the interested reader to these for further details. By \eqref{eq:berry_s_method} we may write
\begin{equation}
\label{eq:berry_s_method_1}
2\pi\mathbb{E}[|v_1||v_2|]=
\iint_{\R^{2}_{+}}
[f_{0,0}(U,V)-f_{t,0}(U,V)-f_{0,s}(U,V)+f_{t,s}(U,V)]
\frac{dtds}{(ts)^{3/2}}
\end{equation}
where, setting $v=(v_1,v_2)$, $dv=dv_1dv_2$ and $J(t,s)=\br{\begin{matrix} tI_d & 0\\
0 & sI_d\end{matrix}}$,
\begin{align*}
f_{t,s}(U,V)&=
\int_{\R^{d}\times\R^{d}}
\frac{1}{\sqrt{(2\pi)^d\det{\Theta}}}e^{-\frac{1}{2}(t|v_1|^2+s|v_2|^2)}
e^{-\frac{1}{2}
\langle v,
\Theta^{-1}
v\rangle
}
dv\\
&=\brb{\det(\Theta)\det\br{J(t,s)+\Theta^{-1}}}^{-1/2}\\
&=\brb{\det(J(t,s))^{1/2}\det(\Theta+J(t,s)^{-1})\det(J(t,s))^{1/2}}^{-1/2}\\
&=\brb{\det\brb{I_{2d}+J(t,s)^{1/2}\Theta J(t,s)^{1/2}}}^{-1/2}\, .
\end{align*}
But,
\[
J(t,s)^{1/2}\Theta J(t,s)^{1/2}=\br{\begin{matrix}
t\Omega & 0 \\ 0 & s\Omega
\end{matrix}}+\br{\begin{matrix}
tU & \sqrt{ts}V\\ \sqrt{ts}V & s U
\end{matrix}}\, .
\]
We wish to expand $f_{t,s}(U,V)$ in powers of $U$ and $V$. To do so, we define $Q_{t,s}(U,V)$ as
\[
Q_{t,s}(U,V):=\br{I_{2d}+\begin{pmatrix}
t\Omega & 0 \\ 0 & s\Omega
\end{pmatrix}}^{-1/2}\begin{pmatrix}
tU & \sqrt{ts}V \\ \sqrt{ts}V & sU
\end{pmatrix}\br{I_{2d}+\begin{pmatrix}
t\Omega & 0 \\ 0 & s\Omega
\end{pmatrix}}^{-1/2}\, .
\]
In particular,
\[
f_{t,s}(U,V)=\det\left[I_{2d}
+
\begin{pmatrix}
t\Omega & 0 \\ 0 & s\Omega
\end{pmatrix}\right]^{-1/2}\det(I_{2d}+Q_{t,s}(U,V))^{-1/2}\, .
\]
Since $\Omega$ has positive eigenvalues, the matrices $\left(I_{2d}+
\begin{pmatrix}
t\Omega & 0 \\ 0 & s\Omega
\end{pmatrix}\right)^{-1/2}$ are uniformly bounded from above and below in $(t,s)$. Moreover, their successive derivatives in $t$ and $s$ are uniformly bounded from above. Careful consideration of the block decomposition of $Q_{t,s}(U,V)$ then shows that $\det(I_{2d}+Q_{t,s}(U,V))$ is analytic in the pair $(t,s)$. Moreover, it is easy to see that for each $C<+\infty$ and $N\in\N$, the two following estimates hold:
\begin{align}\label{eq:berry_s_method_2}\notag
\text{The map }(t,s,U,V)\mapsto f_{t,s}(U,V)\text{ has derivatives of order up to }N\\
\text{ which are bounded uniformly for }t,s\in(0,\infty)\text{ and }\|U\|_\infty,\|V\|_\infty\leq C\, .
\end{align}
and
\begin{align}\label{eq:berry_s_method_3}\notag
\text{The map }(t,s)\mapsto Q_{t,s}(U,V)\text{ has derivatives of order up to }N\\\notag
\text{ which, uniformly for }t,s\in(0,\infty)\text{ and }\|U\|_\infty,\|V\|_\infty\leq C\, ,\\
\text{are bounded by }C(\|U\|_\infty+\|V\|_\infty)\, .
\end{align}
These observations reduce the problem of expanding the integral from \eqref{eq:berry_s_method_1} in powers of $U$ and $V$ to expanding the integrand pointwise. Expanding $\det(I_{2d}+Q_{t,s}(U,V))^{-1/2}$ in powers of traces of powers of $Q_{t,s}(U,V)$ (using for instance the Girard-Waring formula for the coefficients of the characteristic polynomial of a matrix), we deduce that, for some universal constant $a_1,a_2,a_3,a_4\in\R$,
\begin{multline*}
f(t,s)=\det\left[I_{2d}
+
\begin{pmatrix}
tD & 0 \\ 0 & sD
\end{pmatrix}\right]^{-1/2}\\
\times\Big[1-\frac{1}{2}\tr(Q_{t,s})+\frac{1}{4}\tr(Q_{t,s}^2)+a_1\tr(Q_{t,s})^2+a_2\tr(Q_{t,s})\tr(Q_{t,s}^2)\\
+a_3\tr(Q_{t,s})^3+a_4\tr(Q_{t,s}^3)\Big]+g(t,s)=:\sum_{i=1}^7 h^i_{t,s}(U,V) + g(t,s)
\end{multline*}
where, by \eqref{eq:berry_s_method_2} and \eqref{eq:berry_s_method_3}, $g$ satisfies
\[
g(t,s)-g(t,0)-g(0,s)+g(0,0)=O\left(\min(t,1)\min(1,s)(\|U\|_\infty^4+\|V\|_\infty^4)\right)
\]
All that remains is to study each of the integrals
\[
A_i=\int_0^\infty\int_0^\infty h^i_{t,s}(U,V)-h^i_{t,0}(U,V)-h_{0,s}^i(U,V)+h^i_{0,0}(U,V)\frac{dtds}{(ts)^{3/2}}\, .
\]
We study them one by one as they each present some slight specificities:
\begin{itemize}
\item The first term, $h^1_{t,s}(U,V)=\det\left[I_{2d}
+
\begin{pmatrix}
t\Omega & 0 \\ 0 & s\Omega
\end{pmatrix}\right]^{-1/2}$ can be written as $\det(\Omega_t)\det(\Omega_s)$ where $\Omega_t=(I_d+t\Omega)^{-1/2}$ so that
\begin{equation}\label{eq:berry_s_method_h_1}
A_1=\left(\int_0^\infty t^{-3/2}(\det(\Omega_t)-1)dt\right)^2=\Upsilon(\Omega)^2\, .
\end{equation}
\item We can write $h^2_{t,s}(U,V)$ as
\[
h^2_{t,s}(U,V)=-\frac{1}{2}\det(\Omega_t)\det(\Omega_s)\left(t\tr(\Omega_tU\Omega_t)+s\tr(\Omega_sU\Omega_s)\right)=-\frac{1}{2}\det(\Omega_t)\det(\Omega_s)\left(t\tr(U\Omega_t^2)+s\tr(U\Omega_s^2)\right)
\]
so
\[
A_2=-\int_0^\infty t^{-3/2}(\det(\Omega_t)-1)dt\int_0^\infty t^{-1/2}\det(\Omega_t)\tr(U\Omega_t^2)dt\, .
\]
Recalling the definition of the integral $\Psi(\Omega)$ \eqref{eq:upsilon_psi_definition}, we deduce that
\begin{equation}\label{eq:berry_s_method_h_2}
A_2=\Upsilon(\Omega)\tr(U\Psi(\Omega))\, .
\end{equation}
\item The third term can be written as
\[
h^3_{t,s}(U,V)=\frac{1}{4}\det(\Omega_t)\det(\Omega_s)\Big(t^2\tr(\Omega_tU\Omega_t^2U\Omega_t)s^2\tr(\Omega_sU\Omega_s^2U\Omega_s)+2ts\tr(\Omega_tV\Omega_s^2V\Omega_t)\Big)\, .
\]
Hence,
\begin{multline*}
\int_0^\infty\int_0^\infty h^1_{t,s}(U,V)-h^1_{t,0}(U,V)-h_{0,s}^1(U,V)+h^1_{0,0}(U,V)\frac{dtds}{(ts)^{3/2}}\\
=\frac{1}{2}\int_0^\infty t^{-3/2}(\det(\Omega_t)-1)dt\int_0^\infty t^{1/2}\tr(\Omega_tU\Omega_t^2U\Omega_t)dt\\
+\frac{1}{2}\int_0^\infty\int_0^\infty\det(\Omega_t)\det(\Omega_s)\tr(\Omega_tV\Omega_s^2V\Omega_t)(ts)^{-1/2}dtds\, .
\end{multline*}
The first term is $O(\|U\|_\infty^2)$ so, using also the cyclicty of the trace, we obtain
\begin{align}\label{eq:berry_s_method_h_3}
A_3&=\frac{1}{2}\int_0^\infty\int_0^\infty\det(\Omega_t)\det(\Omega_s)\tr(V\Omega_s^2V\Omega_t^2)\frac{dtds}{(ts)^{1/2}}+O(\|U\|_\infty^2)\\\nonumber
&=\frac{1}{2}\tr(V\Psi(\Omega)V\Psi(\Omega))+O(\|U\|_\infty^2)\, .
\end{align}
\item The four remaining terms can be bounded as follows. For the fourth and sixth terms, the square and cube of the trace of $Q_{t,s}(U,V)$ involve only terms of order two and three in $U$ so, reasoning as in the previous point,
\begin{equation}\label{eq:berry_s_method_h_4_6}
A_4,A_6\ll\|U\|_\infty^2\, .
\end{equation}
The fifth term and $\tr(Q_{t,s}(U,V)^3)$ in the seventh term involve only terms of order three in $U$ or terms of order one in $U$ and two in $V$. so, as before
\begin{equation}\label{eq:berry_s_method_h_5_7}
A_5,A_7\ll\|U\|_\infty^2+\|U\|_\infty\|V\|_\infty^2\ll\|U\|_\infty^2+\|V\|_\infty^4\, .
\end{equation}
\end{itemize}
All in all, from equations \eqref{eq:berry_s_method_h_1}, \eqref{eq:berry_s_method_h_2}, \eqref{eq:berry_s_method_h_3}, \eqref{eq:berry_s_method_h_4_6} and \eqref{eq:berry_s_method_h_5_7}, we deduce that
\begin{multline*}
\int_0^\infty\int_0^\infty f(t,s)-f(t,0)-f(0,s)+f(0,0)\frac{dtds}{(ts)^{3/2}}\\
=\Upsilon(\Omega)^2+\Upsilon(\Omega)\tr(U\Psi(\Omega))+\frac{1}{2}\tr(V\Psi(\Omega)V\Psi(\Omega))+O(\|U\|_\infty^2)+O(\|V\|_\infty^4)\, .
\end{multline*}
By \eqref{eq:berry_s_method_1} we reach the desired result.
\end{proof}

\section{Nodal volume distribution: Proof of Proposition \ref{prop:second_chaos}}
\label{s:chaos}

In the present section, we prove Proposition \ref{prop:second_chaos}. To this end, we will use the Wiener Chaos expansion of the volume functional we are studying. Though we will recall the necessary definitions, we refer the reader to \cite{mprw00,cammar} for a more thorough account, and to \cite{noupec} for the general theory.\\

Let $m\in\N$. The field $(F_m(x))_{x\in\T^d}$ from \eqref{eq:field_definition} is measurable with respect to the random variables $\zeta_\lambda$ for $\lambda\in\Lambda_m$, defined on some underlyings probability space $(\Xi,\mathcal{F},\mathbb{P})$. Following \cite{mprw00,cammar} we define the space $\mathbf{A}_m$ as the closure in $L^2(\mathbb{P})$ of the space of linear combination of the random variables $F_m(x)$ for $x\in\T^d$. The space $\mathbf{A}_m$ is a (real) Gaussian Hilbert subspace of $L^2(\mathbb{P})$.\\

We define $C_m(0)$ be the subspace of $L^2(\mathbb{P})$ of constant random variables. Then, by induction, for each $k\in\N$, we define $C_m(k+1)$ as the orthogonal of $C_m(k)$ in the space of polynomials of degree at most $k+1$ in the elements of $\mathbf{A}_m$ for the $L^2$ scalar product. The space $C_m(k)$ is the \textit{$k$-th Wiener chaos} $C_m(k)$ associated with $\mathbf{A}_m$. In particular, any $L^2(\mathbb{P})$ random variable measureable with respect to the elements of $\mathbf{A}_m$ belongs to the closure of
\begin{equation}\label{eq:wiener_chaos}
\bigoplus_{k\geq 0} C_m(k)
\end{equation}
and the terms of this sum are orthogonal. The decomposition of a random variable along this sum is called the \textit{Wiener chaos expansion}. For each $q\in\N$, let $H_q(t)=(-1)^q\gamma(t)^{-1}\frac{d^q}{dt^q}\gamma(t)$ be the $q$-th Hermite polynomial, where $\gamma(t)=\frac{1}{\sqrt{2\pi}}e^{-\frac{1}{2}t^2}$ is the standard Gaussian density. In particular, $H_0(t)=1$ and $H_2(t)=t^2-1$. Let $a_1,\dots,a_{\calN_m}\in\mathbf{A}_m$ form an orthonormal basis $\mathbf{A}_m$. Then, an orthonormal basis of $C_m(k)$ is given by the collection of random variables
\[
\frac{1}{\sqrt{q_1!q_2!\dots q_{\calN_m}!}}H_{q_1}(a_1)\cdot\dots\cdot H_{q_{\calN_m}}(a_{\calN_m})
\]
where $q_1,\dots,q_{\calN_m}\in\N$ are any integers satisfying $q_1+\dots+q_{\calN_m}=k$.\\

In this section and the next, we will study the decomposition of the random variable $\calV_m$ along \eqref{eq:wiener_chaos}. More precisely, for each $k\in\N$, we denote by $\calV_m[k]$ the orthogonal projection of $\calV_m$ onto $C_m(k)$. Note that $\calV_m\in L^2(\mathbb{P})$ as explained in section 4.2 of \cite{cammar}. Hence,
\[
\calV_m=\sum_{k\geq 0}\calV_m[k]\, .
\]

Given $\xi=(\xi_1,\dots,\xi_d)$ a standard Gaussian vector in $\R^d$ and $\Omega$ a $d\times d$ symmetric matrix, we write the Wiener expansion of $\langle\xi,\Omega\xi\rangle^{1/2}$ as
\[
\langle\xi,\Omega\xi\rangle^{1/2}=\sum_{q=(q_1,\dots,q_d)\in\N}\alpha_q(\Omega)\frac{1}{\sqrt{q!}}H_q(\xi)
\]
where $q!:=q_1!\dots q_d!$ and $H_q(\xi):=H_{q_1}(\xi_1)\dots H_{q_d}(\xi_d)$.\\

The coefficients $\alpha_q(I_d)$ have been computed in Appendix A.2 of \cite{cammar}. We will use the following facts about the $\alpha_q(\Omega)$ coefficients.
\begin{lemma}\label{lemma:computing_chaos}
Let $\Omega$ be a $d\times d$ positive definite symmetric matrix. Let $v$ be an eigenvector of $\Omega$ of norm one with eigenvalue $a>0$. Then,
\[
\E\left[\langle\xi,\Omega\xi\rangle^{1/2}\right]=\frac{1}{\sqrt{2\pi}}\int_0^\infty 1-\det(I_d+t\Omega)^{-1/2}\frac{dt}{t^{3/2}}\, .
\]
and
\[
\E\left[\langle\xi,\Omega\xi\rangle^{1/2}(\langle \xi,v\rangle^2-1)\right]=\frac{1}{\sqrt{2\pi}}\int_0^\infty\det(I_d+t\Omega)^{-1/2}(1+ta)^{-1}\frac{adt}{t^{1/2}}\, .
\]
Let $q\in\N^d$. If $|q|$ is odd then $\alpha_q(\Omega)=0$. If $\Omega$ is diagonal and there exists $l\in\{1,\dots,d\}$ such that $q_l$ is odd then $\alpha_q(\Omega)=0$.
\end{lemma}
\begin{proof}
The last two points follow respectively from the parity of $\xi\mapsto\langle\xi,\Omega\xi\rangle^{1/2}$ and, if $\Omega$ is diagonal, the invariance of this same function by changes of signs of the coordinates of $\xi$. By \eqref{eq:berry_s_method},
\[
\E\left[\langle\xi,\Omega\xi\rangle^{1/2}\right]=\frac{1}{\sqrt{2\pi}}\int_0^\infty 1-\E\left[e^{-\frac{1}{2}t\langle x,\Omega x\rangle}\right]\frac{dt}{t^{3/2}}=\frac{1}{\sqrt{2\pi}}\int_0^\infty 1-\det(I_d+t\Omega)^{-1/2}\frac{dt}{t^{3/2}}\, .
\]
Let us assume for simplicity that $v=e_1$ so that $\langle\xi,v\rangle=\xi_1$. Moreover, since $v=1$, $\Omega$ is equal to some matrix $\hat{\Omega}$ of size $d-1\times d-1$ with an added row and column at the top left whose only nonzero coefficient is $\Omega_{11}=a$. Let $\hat{\xi}=(\xi_2,\dots,\xi_d)$. Then, for each $t\geq 0$,
\[
\E\left[\xi_1^2e^{-\frac{1}{2}t\langle \xi,\Omega \xi\rangle}\right]=\E\left[\xi_1^2e^{-\frac{1}{2}tax_1^2}e^{-\frac{1}{2}t\langle \hat{\xi},\hat{\Omega} \hat{\xi}\rangle}\right]=(1+ta)^{-3/2}\det(1+t\hat{\Omega})^{-1/2}=(1+ta)^{-1}\det(1+t\Omega)^{-1/2}\, .
\]
Thus, starting as we did above,
\begin{align*}
\E\left[\langle\xi,\Omega\xi\rangle^{1/2}(\xi_1^2-1)\right]&=\frac{-1}{\sqrt{2\pi}}\int_0^\infty\E\left[(\xi_1^2-1)e^{-\frac{1}{2}t\langle x,\Omega x\rangle}\right]\frac{dt}{t^{3/2}}\\
&=\frac{-1}{\sqrt{2\pi}}\int_0^\infty (1+ta)^{-1}\det(1+t\Omega)^{-1/2}-\det(1+t\Omega)^{-1/2}\frac{dt}{t^{3/2}}\\
&=\frac{1}{\sqrt{2\pi}}\int_0^\infty\det(I_d+t\Omega)^{-1/2}(1+ta)^{-1}\frac{adt}{t^{1/2}}\, .
\end{align*}
\end{proof}

The Wiener chaos expansion of the volume may be computed in terms of the $\alpha_q$. In particular, we have the following result.

\begin{lemma}[Appendix A of \cite{cammar}]\label{lemma:volume_chaos}
Let $P$ be an orthogonal matrix of size $d\times d$. Then,
\[
\calV_m[2k]=\sqrt{\frac{L_m}{2\pi}}\sum_{q\in\N^d,\, |q|\leq 2k} \frac{H_{2k-|q|}(0)}{\sqrt{(2k-|q|)!q!}}\alpha_q(P\Omega_mP^{-1})\int_{\T^d} H_{2k-|q|}(F_m(x))H_q\left(P(L_m\Omega_m)^{-1/2}\nabla F_m(x)\right)dx\, .
\]
\end{lemma}
\begin{proof}
In \cite{cammar}, the author shows that, as $\eps\to+\infty$, the random variable $\calV_m^\eps=\int_{\T^d}\frac{1}{2\eps}\mathbf{1}_{[|F_m(x)|\leq\eps]}\|\nabla F_m(x)\|dx$ converges in $L^2$ to $\calV_m$. In particular, its chaos expansion converges to that of $\calV_m$. The author of \cite{cammar} then computes the expansion of $\calV_m^\eps$ for a fixed $\eps>0$ by arguing as follows. For a fixed $x\in\T^d$, $F(x)$ and $\nabla F_m(x)$ are independent so the decomposition of $\frac{1}{2\eps}\mathbf{1}_{[|F_m(x)|\leq\eps]}\|\nabla F_m(x)\|$ can be easily computed from the decomposition of the two factors. The $k$-th chaos of the factor $\frac{1}{2\eps}\mathbf{1}_{[|F_m(x)|\leq\eps]}$ converges, as $\eps\to 0$, to
\[
\lim_{\eps\to 0}\frac{1}{2\pi}\mathbf{1}_{[|F_m(x)|\leq\eps]}[k]=\frac{1}{\sqrt{2\pi}}H_k(0)H_k(F_m(x))\, .
\]
The only difference in our situation is that $\nabla F_m(x)$ is not proportional to a standard Gaussian vector. Instead, we write
\[
\|\nabla F_m(x)\|=\sqrt{L_m}\langle P^{-1}\xi,\Omega_m P^{-1}\xi\rangle^{1/2}
\]
where $\xi=P(L_m\Omega_m)^{-1/2}\nabla F_m(x)$ is a standard Gaussian vector by Definition \eqref{def:omega}. Thus, multiplying the two expansions and integrating over $x$, we get, for each $k\in\N$,
\[
\calV_m[2k]=\sqrt{\frac{L_m}{2\pi}}\sum_{|q|\leq 2k}\frac{H_{2k-q}(0)}{\sqrt{(2k-|q|)!q!}}\alpha_q(P\Omega_mP^{-1})\int_{\T^d}H_{2k-|q|}(F_m(x))H_q\left(P(L_m\Omega_m)^{-1/2}\nabla F_m(x)\right)dx
\]
as announced.
\end{proof}

\begin{proof}[Proof of Proposition \ref{prop:second_chaos}]
Let $v_1,\dots,v_d$ be an orthonormal basis or eigenvectors of $\Omega_m$ with eigenvalues $\omega_1,\dots,\omega_d$ and let $P$ be the orthogonal matrix whose rows are the vectors $v_j$. Then, for each $x\in\T^d$, the random variables $a_j(x)=(L_m\omega_j)^{-1/2}\langle \nabla F_m(x),v_j\rangle$ for $j=1,\dots,d$ are independent standard normals independent from $a_{d+1}(x)=F_m(x)$. Applying Lemma \ref{lemma:volume_chaos} with the matrix $P$ and $k=1$ we deduce that 
\begin{align*}
\calV_m[2]&=\sqrt{\frac{L_m}{2\pi}}\sum_{q\in\N^d,\, |q|\leq 2} \frac{H_{2-|q|}(0)}{\sqrt{(2-|q|)!q!}}\alpha_q(P\Omega_mP^{-1})\int_{\T^d} H_{2-|q|}(a_{d+1}(x))H_q\left(a_1(x),\dots,a_d(x)\right)dx\\
&=\sqrt{\frac{L_m}{4\pi}}\left[-\alpha_0(P\Omega_m P^{-1})\int_{\T^d} F_m(x)^2-1dx+\sum_{j=1}^d\alpha_{2(\delta_{j,i})_i}(P\Omega_m P^{-1})\int_{\T^d} a_j(x)^2-1dx\right]\\
&=\sqrt{\frac{L_m}{4\pi}}\frac{1}{\calN_m}\sum_{\lambda\in\Lambda_m}\left[\sum_{j=1}^d\alpha_{2(\delta_{j,i})_i}(P\Omega_m P^{-1})(\omega_j^{-1}\langle\lambda_j,v_j\rangle^2|\zeta_\lambda|^2-1)-\alpha_0(P\Omega_mP^{-1})(|\zeta_\lambda|^2-1)\right]\, .
\end{align*}
By Lemma \ref{lemma:computing_chaos}, since $P\Omega_m P^{-1}$ is diagonal, for each $j\in\{1,\dots,d\}$,
\[
\alpha_{2(\delta_{j,i})_i}(P\Omega_m P^{-1})=\frac{1}{\sqrt{2\pi}}\int_0^\infty\det(I_d1+t\Omega_m)^{-1/2}(1+t\omega_j)^{-1}\frac{\omega_jdt}{t^{1/2}}
\]
Hence,
\[
\sum_{j=1}^d\alpha_{2(\delta_{j,i})_i}(P\Omega_m P^{-1})(\omega_j^{-1}\langle\lambda_j,v_j\rangle^2|\zeta_\lambda|^2-1)=\Upsilon(\Omega_m)\Xi_m(\lambda)|\zeta_\lambda|^2-\tr(\Psi(\Omega_m)\Omega_m)\, .
\]
Now, either by direct computation, or using the fact that $\calV_m[2]$ is centered by definition, conclude that
\begin{equation}\label{eq:2_chaos_1}
\frac{1}{\calN_m}\sum_{\lambda\in\Lambda_m}\sum_{j=1}^d\alpha_{2(\delta_{j,i})_i}(P\Omega_m P^{-1})(\omega_j^{-1}\langle\lambda_j,v_j\rangle^2|\zeta_\lambda|^2-1)=\frac{1}{\calN_m}\sum_{\lambda\in\Lambda_m}\Upsilon(\Omega_m)\Xi_m(\lambda)(|\zeta_\lambda|^2-1)\, .
\end{equation}
On the other hand, by Lemma \ref{lemma:computing_chaos},
\[
\alpha_0(P\Omega_mP^{-1})=\alpha_0(\Omega_m)=\Upsilon(\Omega_m)\, .
\]
Using this observation and \eqref{eq:2_chaos_1} in our initial computation, we deduce that
\[
\calV_m[2]=\sqrt{\frac{L_m}{2\pi}}\Upsilon(\Omega_m)\frac{1}{\calN_m}\sum_{\lambda\in\Lambda_m}\left(\Xi_m(\lambda)-1\right)(|\zeta_\lambda|^2-1)
\]
as announced. In particular, since the $\zeta_\lambda$ are independent standard Gaussians,
\[
\textup{Var}(\calV_m[2])=\frac{L_m}{4\pi}\Upsilon(\Omega_m)^2\frac{1}{\calN_m^2}\sum_{\lambda\in\Lambda_m}\left(\Xi_m(\lambda)-1\right)^2=\frac{\pi}{\calN_m^2}\E[\calV_m]^2\sum_{\lambda\in\Lambda_m}\left(\Xi_m(\lambda)-1\right)^2\, .
\]
In the last equality we used Proposition \ref{prop:expectation}.
\end{proof}

\section{Ruling out Berry cancellation for certain ellipsoids}\label{s:anticancellation}

In this section we assume that $p(x)=ax_1^2+x_2^2+x_3^2$ where $a$ is a large positive integer. Our aim is to prove that for all large enough values of $a$, Berry cancellation does not occur. In other words, we will prove the following result. Recall the notations $\Omega$ from \eqref{eq:omega_def}, $\Psi$ from \eqref{eq:upsilon_psi_definition} and $\Sigma_p=\{x\in\R^3\, :\, p(x)=1\}$.

\begin{prop}\label{prop:anticancellation}
Assume that $p(x)=ax_1^2+x_2^2+x_3^2$. There exists $a_0\in\N$ such that if $a\geq a_0$ then $\Psi(\Omega)^{-1/2}\Sigma_p$ is not a sphere.
\end{prop}

In order to prove this proposition, we first observe that, with this choice of $p$, $\Sigma_p$ is left invariant by the transformations $x_i\mapsto -x_i$ so that $\Omega$ is diagonal. In particular, so is $\Psi(\Omega)$ and $\Psi(\Omega)^{-1/2}\Sigma_p$ is given by the equation
\[
\left\{a\Psi(\Omega)_{11}y_1^2+\Psi(\Omega)_{22}y_2^2+\Psi(\Omega)_{33}y_3^2=1\right\}\, .
\]
Hence, in order to prove Proposition \ref{prop:anticancellation}, it suffices to show that $a\Psi(\Omega)_{11}\neq\Psi(\Omega)_{22}$, say. By definition of $\Psi$, for $i=1,2,3$, we have
\begin{equation}\label{eq:anticancellation_2}
\Psi(\Omega)_{ii}=\int_0^\infty t^{-1/2}(1+\Omega_{ii}t)^{-3/2}\prod_{j\neq i}(1+\Omega_{jj}t)^{-1/2}dt\, .
\end{equation}
Moreover, following the expression for $d\sigma_p$ given by Definition \ref{def:equidistribution_measure}, we get, for $i=1,2,3$,

\begin{equation}\label{eq:anticancellation_1}
\Omega_{ii}=\frac{1}{|S^2|}\int_{S^2}\frac{\omega_i^2d\omega}{a\omega_1^2+\omega_2^2+\omega_3^2}\, .
\end{equation}

\begin{proof}[Proof of Proposition \ref{prop:anticancellation}]
The preceding discussion reduces the proof to showing that, for all $a$ large enough, $a\Psi(\Omega)_{11}\neq \Psi(\Omega)_{22}$, where $\Psi(\Omega)_{ii}$ are expressed in terms of $a$ by \eqref{eq:anticancellation_2} and \eqref{eq:anticancellation_1} above. We begin by estimating the behaviour coefficients $\Omega_ii$ as $a\to+\infty$.
\begin{claim}\label{cl:anticancellation_1}
Uniformly for all $a\geq 2$,
\[
\Omega_{22}\asymp a^{-1/2}\, .
\]
\end{claim}
\begin{proof}
Applying the change of variables $(\omega_1,\omega_2,\omega_3)=(t,\sqrt{1-t^2}\cos(\theta),\sqrt{1-t^2}\sin(\theta))$ we get
\begin{align*}
\int_{S^2}\frac{\omega_2^2d\omega}{a\omega_1^2+\omega_2^2+\omega_3^2}&=\int_{-1}^1\int_0^{2\pi}\frac{(1-t^2)\cos(\theta)^2}{at^2+1-t^2}(1-t^2)^{1/2}d\theta dt\\
&=\pi\int_{-1}^1\frac{(1-t^2)^{3/2}}{at^2+1-t^2} dt\, .
\end{align*}
As for the integral, on the one hand, for some universal constant $c>0$,
\[
\int_{-1}^1\frac{(1-t^2)^{3/2}}{at^2+1-t^2} dt\geq c \int_{-1/2}^{1/2}\frac{dt}{at^2+1}\sim c a^{-1/2}\int_{-\infty}^\infty \frac{ds}{s^2+1}
\]
and on the other hand,
\[
\int_{-1}^1\frac{(1-t^2)^{3/2}}{at^2+1-t^2} dt\leq \int_{-1}^1\frac{dt}{(a-1)t^2+1}\sim a^{-1/2}\int_{-\infty}^\infty \frac{ds}{s^2+1}\, .
\]
All in all, $\Omega_{22}\asymp a^{-1/2}$ as announced.
\end{proof}
To estimate $\Omega_{11}$ observe that $a\Omega_{11}+\Omega_{22}+\Omega_{33}=1$ and that $\Omega_{22}=\Omega_{33}$ so that, by Claim \ref{cl:anticancellation_1},
\begin{equation}\label{eq:anticancellation_3}
\Omega_{11}\sim a^{-1}\, .
\end{equation}
We can now estimate $\Psi(\Omega)_{11}$ and $\Psi(\Omega)_{22}$. In both cases we apply the change of variables $s=\Omega_{22}t$, which yields
\begin{align*}
\Psi(\Omega)_{11}&=\Omega_{22}^{-1/2}\int_0^\infty \frac{ds}{s^{1/2}(1+(\Omega_{11}/\Omega_{22})s)^{3/2}(1+s)}\\
\Psi(\Omega)_{22}&=\Omega_{22}^{-1/2}\int_0^\infty \frac{ds}{s^{1/2}(1+s)^2(1+(\Omega_{11}/\Omega_{22})s)^{1/2}}\, .
\end{align*}
Since $\Omega_{11}\ll \Omega_{22}$ by Claim \ref{cl:anticancellation_1} and \eqref{eq:anticancellation_3},
\begin{align*}
c\int_0^\infty\frac{ds}{s^{1/2}(1+s)^{5/2}}&\leq \Omega_{22}^{1/2}\Psi(\Omega)_{11}\leq \int_0^\infty\frac{ds}{s^{1/2}(1+s)}\\
c\int_0^\infty\frac{ds}{s^{1/2}(1+s)^{5/2}}&\leq \Omega_{22}^{1/2}\Psi(\Omega)_{22}\leq \int_0^\infty\frac{ds}{s^{1/2}(1+s)^2}
\end{align*}
for some universal $c>0$. All in all, $\Psi(\Omega)_{11}\asymp\Psi(\Omega)_{22}\asymp a^{1/4}$ as $a\to+\infty$. In particular, for all $a$ large enough, $a\Psi(\Omega)_{11}>\Psi(\Omega)_{11}$ and the proof is complete.
\end{proof}
\appendix 

\section{On lattice point counts on ellipses}\label{s:upper_bd_proof}

In this section we provide a proof for Lemma \ref{lemma:upper_bd_ellipses}. We begin with the following proposition, which is an immediate consequence of the discussion of section II.A of \cite{cilcor}.
\begin{prop}[\cite{cilcor}, section II.A]\label{prop:upper_bd_ellipses}
Let $p$ be a bivariate quadratic form with integer coefficients. Then, for each $m\in\N$, $m\geq 1$, the number $\calN(m;p)$ of points $(x,y)\in\Z^2$ such that $p(x,y)=m$ satisfies the following bound. Consider the prime decomposition of $m$:
\[
m=\prod_{p\in\N,\textup{ prime}}p^{\alpha_p}\, .
\]
Then,
\[
\calN(m;p)\leq 6\prod_{p\in\N,\textup{ prime}}(1+\alpha_p)\, .
\]
If, moreover, $\calN(m;p)>0$, then
\[
\calN(m;p)\geq \prod_{p\in\N,\textup{ prime}}(1+\alpha_p)\, .
\]
\end{prop}
Let us now prove Lemma \ref{lemma:upper_bd_ellipses} using Proposition \ref{prop:upper_bd_ellipses}.
\begin{proof}[Proof of Lemma \ref{lemma:upper_bd_ellipses}]
We start by fixing $\eps>0$. Then, we define $A=A(\eps)\in\N$ such that for each $\alpha\geq A$ and each prime $p$ (the most restrictive case being $p=2$),
\begin{equation}\label{eq:upper_bd_ellipses_1}
(1+\alpha)\leq p^{\eps\alpha}\, .
\end{equation}
Given $m\in\N$, $m\geq 1$, we write its prime decomposition as follows
\[
m=\prod_{p\in\N,\textup{ prime}}p^{\alpha_p}\, .
\]
We then define
\[
m_1=\prod_{p,\, \alpha_p< A}p^{\alpha_p}\, ;\ m_2=\prod_{p,\, \alpha_p\geq A}p^{\alpha_p}\textup{ so that }m_1m_2=m\, .
\]
Then, by \eqref{eq:upper_bd_ellipses_1}
\begin{equation}\label{eq:uper_bd_ellipses_2}
\prod_{p,\, \alpha_p\geq A}(1+\alpha_p)\leq m_2^{\eps}\, .
\end{equation}
On the other hand, writing $T$ for the number of prime factors $p$ for which $1\leq \alpha_p< A$, we get
\[
\prod_{p\, \alpha_p<A}(1+\alpha_p)\leq A^T
\]
and
\[
m_1\geq T!\, . 
\]
But, having fixed $A$ and $\eps>0$, $A^T=O((T!)^\eps)$ so that
\begin{equation}\label{eq:upper_bd_ellipses_3}
\prod_{p,\, \alpha_p<A}(1+\alpha_p)=O(m_2^\eps)\, .
\end{equation}
But since $m_1m_2=m$, \eqref{eq:uper_bd_ellipses_2} and \eqref{eq:upper_bd_ellipses_3} yield the desired result.
\end{proof}
\section{On equidistribution of lattice points on ellipses}\label{s:2D_equidistribution}

In this section we rephrase a result from \cite{dias00} in terms that are closer to the topic of this paper. In particular, we will denote by $\mathfrak{D}=\{3,4,7,8,\dots,3315,5460\}$ the set of numbers listed in section 4.1 of \cite{dias00}.\\

As in \cite{dias00} we will work in dimension $d=2$ throughout this section. In particular, the quadratic form $p$ we consider takes the form
\[
p(x,y)=ax^2+bxy+cy^2
\]
where $a,b,c\in\Z$. Since the values $p$ takes for integer entries must be multiples of the g.c.d. of $(a,b,c)$, we assume in addition that $\textup{gcd}(a,b,c)=1$. Quadratic forms with this property will be called \textit{primitive}.\\

Since $p$ is definite positive, the polynomial $p(x,1)$ has a unique root $\tau\in\C$ with strictly positive imaginary part. The results will be expressed using the following, injective mapping:
\begin{align*}
\Z^2&\rightarrow\C\\
(x,y)&\mapsto \alpha_{x,y}:=a(x-\overline{\tau}y)\, .
\end{align*}
Note that for any $(x,y)\in\Z^2$,
\begin{equation}\label{eq:norm_of_alpha}
|\alpha_{x,y}|^2=p(x,y)\, .
\end{equation}
We wish to state a result estimating the angular equidistribution of the points $\alpha_{x,y}$ of fixed norm for $(x,y)\in\Z^2$. This equidistribution is measured using the following quantities. For each $m\in\N$:
\[
\calN_m=\{(x,y)\in\Z^2\, :\, |\alpha_{x,y}|^2=m\}\, .
\]
Note that by \eqref{eq:norm_of_alpha}, $\calN_m$ coincides with the definition used in the rest of the paper (i.e., the number of lattice points $\lambda\in\Z^d$ such that $p(\lambda)=m$).
\[
\Delta_p(m)=\max\left\{\textup{card}^*\{(x,y)\in\Z^2\, :\, |\alpha_{x,y}|^2=m,\, \textup{arg}(\alpha_{x,y})\in[\theta_1,\theta_2]\}-(\theta_2-\theta_1)\calN_m\, :\, 0\leq\theta_1<\theta_2\leq 2\pi\right\}\, .
\]
Here $\textup{arg}(\alpha)$ denotes the argument of $\alpha$ in $]0,2\pi]$ and $\textup{card}^*$ means that the points $(x,y)\in\Z^2$ such that $\textup{arg}(\alpha_{x,y})\in\{\theta_1,\theta_2\}$ count for $\frac{1}{2}$ instead of $1$.\\

Combining Lemma 5 and Theorem 3 of \cite{dias00}, one obtains the following result.

\begin{thm}\label{thm:dias}
Assume that the positive definite quadratic form $p(x,y)=ax^2+bxy+cy^2$ is primitive and that the discriminant $-\delta=b^2-4ac$ is such that $\delta\in\mathfrak{D}$. Then, for each $\eps>0$ there exists $C=C(p,\eps)<\infty$ such that for each $T>0$, the set $E'_T$ of integers $m\in[0,T]$ such that
\[
\Delta_p(m)\leq C\frac{\calN_m}{(\log T)^{\frac{1}{2}\log(\frac{\pi}{2})-\eps}}
\]
and the set $E_T$ of integers $m\in[0,T]$ for which $\calN_m>0$ satisfy the property that 
\[
\lim_{T\to+\infty}\frac{\textup{Card}(E'_T)}{\textup{Card}(E_T)}=1\, .
\]
\end{thm}

\begin{rem}
As explained by its author, the approach of \cite{dias00} is to adapt the classical equidistribution result by Erd\"os and Hall \cite{erdhal} to the case of general quadratic forms.
\end{rem}

\section{Computation of the expected volume when $\Omega_m$ is diagonal}

In this section, we fix $m\in\N$ and assume that $\Omega_m=\alpha I_d$ for some $\alpha>0$. From Proposition \ref{prop:expectation}, the expectation of the nodal volume is
\[
\E[\calV_m]=\frac{\sqrt{L_m}}{2\pi}\Upsilon(\alpha I_d)
\]
where $\Upsilon$ is defined in \eqref{eq:upsilon_psi_definition}. Our goal is to express $\Upsilon(\alpha I_d)$ and hence $\E[\calV_m]$ in terms of $\alpha$ and $d$ using classical functions.
\begin{align*}
\Upsilon(\alpha I_d)&=\int_0^\infty 1-(1+\alpha t)^{-d/2}t^{-3/2}dt\\
&=\alpha^{1/2}\int_0^\infty 1-(1+s)^{-d/2}s^{-3/2}ds\textup{ by setting }s=\alpha t\\
&=2d\alpha^{1/2}\int_0^\infty (1+u^2)^{-\frac{d+2}{2}}du\textup{ by setting }u^2=s\textup{ and integrating by parts}\, .
\end{align*}
For each $d\in\N$, let $I_d=\int_0^\infty(1+u^2)^{-d/2}du$. Integrating by parts yields the following inductive relation for all $d>0$:
\[
I_{d+2}=\frac{d-1}{d}I_d
\]
from which we deduce that, for all $d>0$,
\[
I_{2d}=I_2\frac{(2(d-1))!}{4^{d-1}(d-1)!^2}\textup{ and }I_{2d+1}=I_3\frac{4^{d-1}(d-1)!^2}{(2d-1)!}\, .
\]
Moreover, setting $u=\tan(\theta)$ yields $I_2=\int_0^\infty \frac{du}{1+u^2}=\frac{\pi}{2}$ and $I_3=\int_0^\infty\frac{du}{(1+u^2)^{3/2}}=\frac{u}{\sqrt{1+u^2}}\Big|_{u=0}^\infty=1$. All in all,
\[
\Upsilon(\alpha I_{2d})=4d\alpha^{1/2}\times\frac{\pi}{2}\times \frac{(2d)!}{4^d d!^2}\textup{ and }\Upsilon(\alpha I_{2d+1})=2(2d+1)\alpha^{1/2}\times\frac{4^{d}d!^2}{(2d+1)!}\, .
\]
Both cases can be summarised as follows:
\begin{equation}\label{eq:upsilon_formula}
\Upsilon(\alpha I_d)=\sqrt{4\pi\alpha}\frac{\Gamma\left(\frac{d+1}{2}\right)}{\Gamma\left(\frac{d}{2}\right)}
\end{equation}
from which we get
\begin{equation}\label{eq:expectation_formula}
\E[\calV_m]=\sqrt{4\pi\alpha}\frac{\Gamma\left(\frac{d+1}{2}\right)}{\Gamma\left(\frac{d}{2}\right)}\times m^{1/2k}\, .
\end{equation}
\addcontentsline{toc}{section}{References}
\bibliographystyle{plain}
\bibliography{bibfile}

\begin{thebibliography}{10}

\bibitem{azawsc}
Jean-Marc Aza{\"{\i}}s and Mario Wschebor.
\newblock {\em Level sets and extrema of random processes and fields}.
\newblock John Wiley \& Sons, Inc., Hoboken, NJ, 2009.

\bibitem{becawi}
Dmitry Beliaev, Valentina Cammarota, and Igor Wigman.
\newblock Two point function for critical points of a random plane wave.
\newblock {\em International Mathematics Research Notices}, 2019(9):2661--2689,
  2019.

\bibitem{benmaf}
Jacques Benatar and Riccardo~W. Maffucci.
\newblock Random waves on $\mathbb{T}^3$: Nodal area variance and lattice point
  correlations.
\newblock {\em International Mathematics Research Notices},
  2019(10):3032--3075, May 2019.

\bibitem{berard_1985}
Pierre {B\'erard}.
\newblock {Volume des ensembles nodaux des fonctions propres du Laplacien}.
\newblock In {\em {S\'eminaire de th\'eorie spectrale et g\'eom\'etrie. Ann\'ee
  1984-1985}}, page~ex. Chamb\'ery: Univ. de Savoie, Fac. des Sciences, Service
  de Math.; St. Martin d'H\`eres: Univ. de Grenoble I, Inst. Fourier, 1985.

\bibitem{berry2}
Michael~V. Berry.
\newblock Statistics of nodal lines and points in chaotic quantum billiards:
  perimeter corrections, fluctuations, curvature.
\newblock {\em Journal of Physics A: Mathematical and General}, 35(13):3025,
  2002.

\bibitem{birc62}
Brian~J. Birch.
\newblock Forms in many variables.
\newblock {\em Proceedings of the Royal Society of London. Series A.
  Mathematical and Physical Sciences}, 265(1321):245--263, 1962.

\bibitem{bombou}
Enrico Bombieri and Jean Bourgain.
\newblock A problem on sums of two squares.
\newblock {\em International Mathematics Research Notices},
  2015(11):3343--3407, 2015.

\bibitem{bopi89}
Enrico Bombieri, Jonathan Pila, et~al.
\newblock The number of integral points on arcs and ovals.
\newblock {\em Duke Math. J}, 59(2):337--357, 1989.

\bibitem{bode15}
Jean Bourgain and Ciprian Demeter.
\newblock The proof of the $l^2$ decoupling conjecture.
\newblock {\em Annals of mathematics}, pages 351--389, 2015.

\bibitem{cammar}
Valentina Cammarota.
\newblock Nodal area distribution for arithmetic random waves.
\newblock {\em Transactions of the American Mathematical Society}, 2019.

\bibitem{cmwsh4}
Valentina Cammarota, Domenico Marinucci, and Igor Wigman.
\newblock On the distribution of the critical values of random spherical
  harmonics.
\newblock {\em The Journal of Geometric Analysis}, 26(4):3252--3324, 2016.

\bibitem{canhan}
Yaiza Canzani and Boris Hanin.
\newblock Local universality for zeros and critical points of monochromatic
  random waves.
\newblock {\em arXiv preprint arXiv:1610.09438}, 2016.

\bibitem{cantot}
Yaiza Canzani and John~A. Toth.
\newblock Nodal sets of {S}chr{\"o}dinger eigenfunctions in forbidden regions.
\newblock {\em Annales Henri Poincar{\'e}}, 17(11):3063--3087, 2016.

\bibitem{ccdn19}
Wouter Castryck, Raf Cluckers, Philip Dittmann, and Kien~Huu Nguyen.
\newblock The dimension growth conjecture, polynomial in the degree and without
  logarithmic factors.
\newblock {\em arXiv preprint arXiv:1904.13109}, 2019.

\bibitem{cheng1}
Shiu-Yuen Cheng.
\newblock Eigenfunctions and nodal sets.
\newblock {\em Commentarii Mathematici Helvetici}, 51(1):43--55, 1976.

\bibitem{chelaa}
Giacomo Cherubini and Niko Laaksonen.
\newblock On the variance of the nodal volume of arithmetic random waves.
\newblock {\em arXiv preprint arXiv:2007.12143}, 2020.

\bibitem{cilcor}
Javier Cilleruelo and Antonio C{\'o}rdoba.
\newblock Trigonometric polynomials and lattice points.
\newblock {\em Proc. Amer. Math. Soc.}, 115(4):899--905, 1992.

\bibitem{cicoel}
Javier Cilleruelo and Antonio C{\'o}rdoba.
\newblock Lattice points on ellipses.
\newblock {\em Duke Mathematical Journal}, 76(3):741--750, 1994.

\bibitem{daesle}
Federico Dalmao, Anne Estrade, and Jos{\'e} Le{\'o}n.
\newblock On 3-dimensional {B}erry's model.
\newblock {\em arXiv preprint arXiv:1912.09774}, 2019.

\bibitem{dias00}
Dimitri Dias.
\newblock The angular distribution of integral ideal numbers with a fixed norm
  in quadratic extensions.
\newblock {\em arXiv preprint arXiv:1404.6271}, 2014.

\bibitem{durusa}
William Duke, Ze{\'e}v Rudnick, Peter Sarnak, et~al.
\newblock Density of integer points on affine homogeneous varieties.
\newblock {\em Duke mathematical journal}, 71(1):143--179, 1993.

\bibitem{elhtot}
Layan El-Hajj and John~A. Toth.
\newblock Intersection bounds for nodal sets of planar {N}eumann eigenfunctions
  with interior analytic curves.
\newblock {\em Journal of Differential Geometry}, 100(1):1--53, 2015.

\bibitem{erdhal}
Paul Erd{\"o}s and Richard~Roxby Hall.
\newblock On the angular distribution of gaussian integers with fixed norm.
\newblock {\em Discrete mathematics}, 200(1-3):87--94, 1999.

\bibitem{fome86}
Oleg~Mstislavovich Fomenko.
\newblock On the uniform distribution of integer points on multidimensional
  ellipsoids.
\newblock {\em Zapiski Nauchnykh Seminarov POMI}, 154:144--153, 1986.

\bibitem{hartshorne}
Robin {Hartshorne}.
\newblock {\em {Algebraic geometry. Corr. 3rd printing}}, volume~52.
\newblock Springer, New York, NY, 1983.

\bibitem{hbanna}
David~R. Heath-Brown.
\newblock The density of rational points on curves and surfaces.
\newblock {\em Annals of mathematics}, 155(2):553--598, 2002.

\bibitem{hbsurv}
David~R. Heath-Brown.
\newblock Counting rational points on algebraic varieties.
\newblock In {\em Analytic number theory}, pages 51--95. Springer, 2006.

\bibitem{iwniec}
Henryk Iwaniec.
\newblock Fourier coefficients of modular forms of half-integral weight.
\newblock {\em Inventiones mathematicae}, 87(2):385--401, 1987.

\bibitem{iwanbk}
Henryk Iwaniec.
\newblock {\em Topics in classical automorphic forms}, volume~17.
\newblock American Mathematical Soc., 1997.

\bibitem{krkuwi}
Manjunath Krishnapur, P{\"a}r Kurlberg, and Igor Wigman.
\newblock Nodal length fluctuations for arithmetic random waves.
\newblock {\em Ann. of Math. (2)}, 177(2):699--737, 2013.

\bibitem{maff3d}
Riccardo~W. Maffucci.
\newblock Nodal intersections for random waves against a segment on the
  3-dimensional torus.
\newblock {\em Journal of Functional Analysis}, 272(12):5218--5254, 2017.

\bibitem{magy02}
{\'A}kos Magyar.
\newblock Diophantine equations and ergodic theorems.
\newblock {\em American journal of mathematics}, 124(5):921--953, 2002.

\bibitem{magybk}
{\'A}kos Magyar.
\newblock On the distribution of solutions to diophantine equations.
\newblock In {\em A panorama of discrepancy theory}, pages 487--538. Springer,
  2014.

\bibitem{mahler}
Kurt Mahler.
\newblock Note on hypothesis {$K$} of {H}ardy and {L}ittlewood.
\newblock {\em Journal of the London Mathematical Society}, 1(2):136--138,
  1936.

\bibitem{mprw00}
Domenico Marinucci, Giovanni Peccati, Maurizia Rossi, and Igor Wigman.
\newblock Non-universality of nodal length distribution for arithmetic random
  waves.
\newblock {\em Geometric and Functional Analysis}, 26(3):926--960, 2016.

\bibitem{noupec}
Ivan Nourdin and Giovanni Peccati.
\newblock {\em Normal approximations with Malliavin calculus: from Stein's
  method to universality}, volume 192.
\newblock Cambridge University Press, 2012.

\bibitem{orruwi}
Ferenc Oravecz, Ze{\'e}v Rudnick, and Igor Wigman.
\newblock The {L}eray measure of nodal sets for random eigenfunctions on the
  torus.
\newblock {\em Annales de l'Institut Fourier}, 58(1):299--335, 2008.

\bibitem{pila95}
Jonathan Pila.
\newblock Density of integral and rational points on varieties.
\newblock {\em Ast{\'e}risque}, 228:183--187, 1995.

\bibitem{rudwi2}
Ze{\'e}v Rudnick and Igor Wigman.
\newblock On the volume of nodal sets for eigenfunctions of the {L}aplacian on
  the torus.
\newblock {\em Ann. Henri Poincar\'e}, 9(1):109--130, 2008.

\bibitem{ruwiye}
Ze\'ev Rudnick, Igor Wigman, and Nadav Yesha.
\newblock Nodal intersections for random waves on the 3-dimensional torus.
\newblock {\em Ann. Inst. Fourier (Grenoble)}, 66(6):2455--2484, 2016.

\bibitem{sarn90}
Peter Sarnak.
\newblock {\em Some applications of modular forms}, volume~99.
\newblock Cambridge University Press, 1990.

\bibitem{totzel}
John~A. Toth and Steve Zelditch.
\newblock Counting nodal lines which touch the boundary of an analytic domain.
\newblock {\em J. Differential Geom.}, 81(3):649--686, 2009.

\bibitem{wispha}
Igor Wigman.
\newblock Fluctuations of the nodal length of random spherical harmonics.
\newblock {\em Communications in Mathematical Physics}, 298(3):787--831, 2010.

\bibitem{yau982}
Shing-Tung Yau.
\newblock Survey on partial differential equations in differential geometry.
\newblock {\em Ann. Math. Studies}, 102:3--70, 1982.

\bibitem{yau993}
Shing-Tung Yau.
\newblock Open problems in geometry.
\newblock In {\em Proc. Symp. Pure Math}, volume~54, pages 1--28, 1993.

\bibitem{zelditch_survey}
Steve {Zelditch}.
\newblock {Eigenfunctions and nodal sets}.
\newblock In {\em {Geometry and topology. Lectures given at the geometry and
  topology conferences at Harvard University, Cambridge, MA, USA, April 29--May
  1, 2011 and Lehigh University, Bethlehem, PA, USA, May 25--27, 2012}}, pages
  237--308. Somerville, MA: International Press, 2013.

\bibitem{zyg}
Antoni Zygmund.
\newblock On {F}ourier coefficients and transforms of functions of two
  variables.
\newblock {\em Studia Math.}, 50:189--201, 1974.

\end{thebibliography}
\Addresses
\end{document}